\documentclass[11pt, twoside, leqno]{amsart}  
\usepackage{lipsum}
\usepackage{amsfonts}
\usepackage{graphicx}
\usepackage{epstopdf}
\usepackage{algorithmic}
\usepackage{calligra}
\usepackage{amsfonts,amsmath,amsthm,amssymb}
\usepackage{mathtools}
\usepackage{hyperref}
\usepackage{autonum}
\usepackage{hhline}
\usepackage{array}
\usepackage{diagbox}
\usepackage{tcolorbox}
\usepackage{mdframed}
\usepackage{multicol}
\usepackage{graphicx}
\usepackage{subcaption}
\usepackage{moreverb}
\usepackage{bbm}
\usepackage[margin=1.32in]{geometry}
\usepackage{todonotes}
\usepackage{scalerel,amssymb}
\allowdisplaybreaks
\usepackage{mathrsfs}  
\usepackage{lineno}
\usepackage{todonotes}
\usepackage{tikz}
\usepackage{pgfplots}
\usetikzlibrary{arrows.meta, positioning, calc, fit, backgrounds}
\usepackage{siunitx}
\usepackage[numbers,sort&compress]{natbib}
\definecolor{mygreen}{HTML}{43a047}
\usepackage{subcaption}
\usepackage{doi}
\usepackage{alphalph}
\usepackage{booktabs}
\usepackage{makecell}
\definecolor{darkgreen}{rgb}{0,0.5,0}

\definecolor{lightblue}{HTML}{4682B4}

\usepackage{enumitem}
\usepackage{dutchcal}
\usepackage[normalem]{ulem}
\usepackage[toc,page]{appendix}
\usepackage{multirow}
\theoremstyle{plain}
\newtheorem{theorem}{Theorem}[section]
\newtheorem{lemma}[theorem]{Lemma}
\newtheorem{proposition}[theorem]{Proposition}
\newtheorem{corollary}[theorem]{Corollary}

\theoremstyle{definition}

\theoremstyle{remark}
\newtheorem{remark}[theorem]{Remark}

\newcommand{\eps}{\varepsilon}

\newcommand{\op}{\overline{p}}

\newcommand{\Om}{\Omega}

\newcommand{\Xp}{\mathcal{X}_p}

\newcommand{\real}{\mathfrak{R}}
\newcommand{\vt}{v_t}
\newcommand{\vtt}{v_{tt}}

\newcommand{\pt}{p_t}

\newcommand{\ptt}{p_{tt}}

\newcommand{\ddt}{\frac{\textup{d}}{\textup{d}t}}

\newcommand{\ca}{\zeta}
\newcommand{\cb}{\xi}
\newcommand{\mh}{h_{\textup{FEM}}}
\newcommand{\pN}{p^N}
\newcommand{\ptN}{p_t^N}
\newcommand{\pttN}{p_{tt}^N}
\newcommand{\pNone}{p^{N-1}}

\newcommand{\vN}{v^N}
\newcommand{\vtN}{v_t^N}
\newcommand{\vNone}{v^{N-1}}
\newcommand{\vtNone}{v_t^{N-1}}
\newcommand{\vttN}{v_{tt}^N}
\newcommand{\vttNone}{v_{tt}^{N-1}}

\newcommand{\phiN}{\phi^N}

\newcommand{\XN}{X_N}

\newcommand{\projN}{\textup{Proj}_{\XN}}

\newcommand{\projNt}{\textup{Proj}_{\tilde{X}_N}}

\newcommand{\errpN}{e^{p,N}}
\newcommand{\errptN}{e_t^{p,N}}
\newcommand{\errpttN}{e_{tt}^{p,N}}

\newcommand{\errvN}{e^{v,N}}
\newcommand{\errvtN}{e_t^{v,N}}
\newcommand{\errvttN}{e_{tt}^{v,N}}

\newcommand{\dt}{\, \textup{d} t}

\newcommand{\dx}{\, \textup{d} x}
\newcommand{\dG}{\, \textup{d} \Gamma}
\newcommand{\dxs}{\, \textup{d}x\textup{d}s}

\newcommand{\intTO}{\int_0^T \int_{\Omega}}

\newcommand{\intT}{\int_0^T}
\newcommand{\intt}{\int_0^t}
\newcommand{\intO}{\int_{\Omega}}
\newcommand{\R}{\mathbb{R}} 
\newcommand{\C}{\mathbb{C}} 
\newcommand{\N}{\mathbb{N}} 
\newcommand{\Linf}{L^\infty(\Omega)}
\newcommand{\Ltwo}{L^2(\Omega)}
\newcommand{\Hone}{H^1(\Omega)}
\newcommand{\Htwo}{H^2(\Omega)}

\newcommand{\LtwoLtwo}{L^2(L^2(\Omega))}

\newcommand{\Hthreehalf}{H^{3/2}(\Omega)}
\newcommand{\LtwoHtwo}{L^2(\Htwo)}
\newcommand{\LtwoHthreehalf}{L^2(\Hthreehalf)}

\newcommand{\Tnorm}[1]{\|#1\|_{\Xp }}

\newcommand{\deltar}{\delta_v}

\makeatletter
\newcommand{\leqnomode}{\tagsleft@true}
\newcommand{\reqnomode}{\tagsleft@false}
\makeatother
\newcommand{\vecn}{\boldsymbol{n}}
\newcommand{\trace}[2]{#2}

\newcommand{\LtwoTLtwo}{L^2(0,T;\Ltwo)}

\newcommand{\vone}{v^{(1)}}
\newcommand{\vtwo}{v^{(2)}}

\newcommand{\Xv}{\mathcal{X}_v}

\newcommand{\dtau}{\, \textup{d}\tau}
\newcommand{\change}[1]{\textcolor{black}{#1}}

\newcommand{\deltap}{\delta_p}

\newcommand{\calL}{\mathcal{L}}
\definecolor{grey}{rgb}{0.5,0.5,0.5}

\newcommand{\teresa}[1]{\textcolor{black}{#1}}

\usetikzlibrary{arrows.meta, quotes}% quotes for "syntax" on edges/tos
\makeatletter
\tikzset{
	node on line/.style={
		to path={
			\pgfextra{%
				\edef\tikz@temp{% rescuing nodes and target for edge
					edge[
					line to, path only, % line to = --, path only = no draw, no fill, …
					every edge quotes/.append style={auto=false},% node *on* the line
					nodes={alias=@nodeonline@}]
					coordinate(@nodeonline@)% fallback coordinate
					\unexpanded\expandafter{\tikz@tonodes}(\tikztotarget)
				}\expandafter
			}\tikz@temp
			-- (@nodeonline@) -- (\tikztotarget)}}}
\makeatother
\tikzset{global scale/.style={
		scale=#1,
		every node/.style={scale=#1}
	}
}
\newcommand{\rv}{r_v}
\newcommand{\rp}{r_p}

\newcommand{\ballr}{\mathbb{B}_{\rp, \rv}}

\newcommand{\ov}{\overline{v}}

\newcommand{\bfn}{{\bf n}}

\newcommand{\pfixedzero}{p^{(0)}}
\newcommand{\pfixedone}{p^{(1)}}
\newcommand{\pfixedN}{p^{(N)}}
\newcommand{\pfixedNnegone}{p^{(N-1)}}
\newcommand{\pfixedNone}{p^{(N_1)}}
\newcommand{\pfixedNtwo}{p^{(N_2)}}
\newcommand{\pfixedNonenegone}{p^{(N_1-1)}}
\newcommand{\pfixedNtwonegone}{p^{(N_2-1)}}

\newcommand{\vfixedzero}{v^{(0)}}
\newcommand{\vfixedone}{v^{(1)}}
\newcommand{\vfixedN}{v^{(N)}}
\newcommand{\vfixedNnegone}{v^{(N-1)}}
\newcommand{\vfixedNone}{v^{(N_1)}}
\newcommand{\vfixedNtwo}{v^{(N_2)}}
\newcommand{\vfixedNonenegone}{v^{(N_1-1)}}
\newcommand{\vfixedNtwonegone}{v^{(N_2-1)}}

\newcommand{\puniqone}{p^{<1>}}
\newcommand{\vuniqone}{v^{<1>}}

\newcommand{\puniqtwo}{p^{<2>}}
\newcommand{\vuniqtwo}{v^{<2>}}

\newcommand{\dGs}{\,\textup{d}\Gamma \textup{d}s}

\newcommand{\errprojNp}{\textup{err}(\tilde{p}^N)}

\newcommand{\errprojNv}{\textup{err}(\tilde{v}^N)}

\newcommand{\Xplow}{Y_p}
\newcommand{\Xvlow}{Y_{v}}
\newcommand{\Xplowzero}{Y^0_p}
\newcommand{\Xvlowzero}{Y^0_{v}}
\newcommand{\Xplowell}{Y^\ell_p}
\newcommand{\Xvlowell}{Y^\ell_{v}}

\newcommand{\pvnorm}[1]{|\!|\!|#1|\!|\!|_{\Xplowzero \times \Xvlowzero}}

\newcommand{\pvnormcap}[1]{|\!|\!|#1|\!|\!|_{\left(\Xplowzero \cap H^{2}(\Hone)\right)\times \Xvlowzero}}

\newcommand{\tpN}{\tilde{p}^N}
\newcommand{\tvN}{\tilde{v}^N}

\newcommand{\tpzero}{\tilde{p}^0}

\newcommand{\fN}{f^N}

\def\fraka{\mathfrak{a}}
\def\frakb{\mathfrak{b}}

\def\pref{p^{\textup{ref}}}
\def\vref{v^{\textup{ref}}}

\def\hFEM{h_{\textup{FEM}}}   
\makeatletter
\@namedef{subjclassname@2020}{\textup{2020} Mathematics Subject Classification}
\makeatother
\title[Multiharmonic algorithms for contrast-enhanced ultrasound]{Multiharmonic algorithms  \\[1mm] for contrast-enhanced ultrasound}


\author{Vanja Nikoli\'c$^\dagger$}  
\thanks{$^\dagger$Department of Mathematics,
	Radboud University,      
	Heyendaalseweg 135,    
	6525 AJ Nijmegen, The Netherlands (\href{vanja.nikolic@ru.nl}{vanja.nikolic@ru.nl})}   
\author{Teresa Rauscher$^\ddag$}
\thanks{$^\ddag$Department of Mathematics and Scientific Computing, 
	University of Graz, 
	Heinrichstra\ss e 36, A-8010 Graz, Austria (\href{teresa.rauscher@uni-graz.at}{teresa.rauscher@uni-graz.at})}
\begin{document}
\vspace*{8mm}
\begin{abstract} 
	Harmonic generation plays a crucial role in contrast-enhanced ultrasound, both for imaging and therapeutic applications.  However, accurately capturing these nonlinear effects is computationally demanding when using traditional time-domain approaches. To address this issue, we develop algorithms based on a time discretization that uses a multiharmonic Ansatz applied to a model that couples the Westervelt equation for acoustic pressure with a {volume-based} approximation of the Rayleigh--Plesset equation for the dynamics of microbubble contrast agents. We first rigorously establish the existence of time-periodic solutions for this Westervelt-ODE system. We then derive a multiharmonic representation of the system under time-periodic excitation and develop iterative algorithms that rely on the successive computation of higher harmonics assuming either real-valued or complex-valued solution fields. In the real-valued setting, we characterize the approximation error in terms of the number of harmonics and a contribution arising from the fixed-point iteration. Finally, we investigate these algorithms numerically and illustrate  how the number of harmonics and the presence of microbubbles  influence the propagation of acoustic waves.
\end{abstract}
\vspace*{-7mm}
\maketitle           
\section{Introduction}
Contrast-enhanced ultrasound has become an important  tool in biomedical applications, with gas-filled microbubbles being used to improve both diagnostic and therapeutic procedures. \teresa{While ultrasound propagation can be described by linear acoustic models in the small-amplitude regime, nonlinear effects arise at higher acoustic pressures, leading to the generation of higher harmonics in the frequency domain.} These nonlinearities are further amplified by microbubble contrast agents, which exhibit strongly nonlinear oscillatory behavior when exposed to ultrasound waves. This delicate  back-and-forth interaction is beneficial for improving resolution in imaging but also for enhancing therapeutic treatments, such as targeted drug delivery; see, for example,~\cite{hoff2001acoustic, stride2009physical, ferrara2007ultrasound} for details. With the rise in the number of applications of contrast-enhanced ultrasound, accurate modeling and efficient simulation in this context have also become prominent research topics; see, e.g.,~\cite{versluis2020ultrasound, blanken2024proteus, doinikov2011review, matalliotakis2023computation}, and the references given therein. \\
\indent The present work builds upon \cite{nikolic2024mathematicalmodelsnonlinearultrasound}, where time-domain mathematical models for ultrasound contrast imaging with microbubbles based on a nonlinear acoustic wave equation coupled to a Rayleigh--Plesset-type ODE \teresa{in a continuum (effective-medium) description of the bubbly mixture} have been derived and investigated in terms of local well-posedness and numerical simulations.  As noted in~\cite{nikolic2024mathematicalmodelsnonlinearultrasound}, a major computational challenge when simulating such systems stems from different time scales on which the wave equation and ODE (nonlinearly) evolve. As a result, straightforward numerical approaches for solving such systems demand using prohibitively small time steps. In this work, we approach this issue by developing \emph{multiharmonic} algorithms for time-periodic solution fields, which offer a potentially more efficient modeling alternative by making use of harmonic expansions; see, e.g.,~\cite{bachinger2005numerical, bachinger2006efficient}, where multiharmonic ideas have been explored for problems arising in electromagnetism and, e.g.,~\cite{Rainer2024nonlinear, kaltenbacher2021periodic, kaltenbacher2024well, groth2021accelerating}, where they have been developed for single-physics acoustic models. Instead of resolving every oscillation in the space-time domain, these methods decompose the field into a sum of harmonics, which can be obtained as solutions of suitable  (in our case) Helmholtz problems and algebraic equations.
These methods are especially promising for simulating real-time ultrasound applications in which the number of harmonics used can be relatively low.

As the starting  time-domain model of contrast-enhanced ultrasound we employ the following coupled nonlinear wave-ODE system:
\begin{equation} \label{West_RPE}
	\boxed{
		\begin{aligned}
			&\ \ p_{tt} - c^2 \Delta p - b \Delta p_t = \eta (p^2)_{tt} + c^2 \rho_0 n_0(x) \vtt 
			+ h(x,t) \quad &&\text{in } \Omega \times (0,T), \\[1mm]
			&\ \ v_{tt} + \delta \omega_0 v_t + \omega_0^2 v  = \ca v^2 + \cb(2v \vtt + \vt^2)- \mu p \quad &&\text{in } \Omega \times (0,T), \\
		\end{aligned} 
	}
\end{equation}
consisting of the damped Westervelt equation \cite{westervelt1963parametric} for the acoustic pressure $p=p(x,t)$ (that is, fluctuations in the background pressure) and an ODE pointwise a.e.~in space for the volume variation  $v=v(x,t)$ of microbubbles. The system in \eqref{West_RPE}  models the interaction of acoustic pressure waves with oscillations of \teresa{microbubbles}.  
\subsection{Modeling background}
\indent In the Westervelt equation, $c>0$ denotes the \teresa{(constant)} speed of sound in the medium, $b>0$ the diffusivity of sound, so that the term $-b \Delta \pt$ introduces strong damping, and $\rho_0>0$ the mass density of the mixture at equilibrium. The nonlinearity coefficient is given by $\eta = \frac{\beta_a}{\rho_0 c^2}$, where $\beta_a$ is the nonlinearity parameter in the medium.  The source of the pressure waves is provided in part through a contribution due to bubble oscillations, modeled by the term $c^2 \rho_0 n_0 v_{tt}$, where $n_0=n_0(x)$ is the bubble number density at equilibrium, and by the source function $h=h(x,t)$.  We allow $n_0$ to vary in space (that is, we assume that $n_0 \in \Linf$), as  this setting is relevant for imaging applications; see, e.g.,~\cite{kaltenbacher_rundell2021}. Moreover, $n_0$ and $h$ regulate the strength of the acoustic source and will be assumed to be small enough in the well-posedness analysis; see Theorem~\ref{thm: existence West-volume} and the discussion in Section~\ref{sec: smallness} on the physical meaning of the smallness assumption. \\
\indent  The ODE in \eqref{West_RPE} captures nonlinear harmonic oscillations driven by the acoustic pressure via the term $-
\mu p$. The total microbubble volume is  given by $V= v_0 + v$, where $v_0$ is the equilibrium volume. The coefficient $\delta = \frac{4 \nu}{\omega_0 R_0^2}$ is the viscous damping coefficient and $\omega_0 = \sqrt{\frac{3 \kappa P_0}{\rho_0 R_0^2}}$ the natural frequency, where $R_0$ is the bubble radius at equilibrium volume $v_0$ (that is, $v_0 = \frac{4 \pi}{3} R_0^3$), $\kappa$ the adiabatic exponent and $P_0$ the ambient pressure in the mixture (so that the total pressure is $P_0 +p$).  
The coefficients appearing on the right-hand side of the ODE are given in terms of equilibrium values by $\mu = \frac{4 \pi R_0}{\rho_0}$, $\ca= \frac{(\kappa + 1) \omega_0^2}{2 v_0}$, and $\cb = \frac{1}{6 v_0}$.  As discussed in~\cite[Ch.\ 5]{hamilton1998nonlinear}, the ODE in \eqref{West_RPE} can be seen as a \teresa{volume-based} approximation of the following Rayleigh--Plesset equation: 
\begin{equation} \label{RP eq}
	\rho_0 \left[ R R_{tt} + \tfrac{3}{2} R_t^2 \right] = p_b - 4 \nu \frac{R_t}{R} - p,
\end{equation}
where $\nu$ is the kinematic viscosity and $p_b$ a constant pressure contribution, which can be derived using the relation $ V= \frac{4 \pi}{3} R^3$ and the adiabatic gas law $\frac{p_b}{P_0} = \left( \frac{v_0}{V}\right)^{\kappa}$. \teresa{Introducing the volume variable and expanding the resulting expression about the equilibrium radius $R_0$, while neglecting higher-order terms, yields a simplified ODE for the volume perturbation $v$. A detailed derivation of this approximation is provided in \cite[Sec.~2.3]{rauscher2025imaging}.} Reformulating the dynamics in terms of $v$ instead of $R$ removes singular terms and results in a more manageable nonlinear structure. Since pressure perturbations couple more directly to volume oscillations, the volume-based formulation is also better suited for harmonic expansions. In contrast, a harmonic expansion in terms of $R$ is expected to lead to higher-order interactions that complicate the isolation of harmonic components due to the cubic dependence of volume on radius. \\  
\indent We equip \eqref{West_RPE} with absorbing-type boundary conditions of the following form: 
\begin{equation}
	\beta p_t + \gamma p + \nabla p \cdot \vecn  = 0 \quad \text{ on } \partial \Omega \times (0, T), \quad  \beta,\, \gamma>0,
\end{equation}
and we are interested in time-periodic solutions that satisfy
\begin{equation}
	\begin{aligned}
		p(0)=p(T), \qquad p_t(0)=p_t(T) \qquad \text{ in } \Omega, \\
		v(0)=v(T), \qquad v_t(0)=v_t(T) \qquad \text{ in } \Omega,
	\end{aligned}
\end{equation}
\teresa{which corresponds to steady-state oscillatory (stable cavitation) regimes rather than transient cavitation.} Although in the existence analysis we do not require the source $h$ to be time periodic, for developing multiharmonic algorithms, we assume that $h$ represents the second time derivative of a $T$-periodic function $g$; see Section~\ref{sec: multiharmonic} for details.
\subsection{Main contributions} The overall purpose of the present work is to establish multiharmonic approximation approaches for time-domain systems in the form of \eqref{West_RPE}, that is, systems that incorporate both nonlinear acoustic propagation and nonlinear microbubble oscillations. To approximate the problem we employ a cut-off Fourier series for the pressure and volume fields:
\begin{equation} \label{cut off approx}
	\begin{aligned}
		p \approx \pN= \real \left\lbrace \sum_{m=0}^{N}  \exp(\imath m \omega t) p_m^N(x)\right\rbrace,  \quad v \approx \vN = \real \left\lbrace \sum_{m=0}^{N}  \exp(\imath m \omega t) v_m^N(x)\right\rbrace,
	\end{aligned}
\end{equation}
with $\omega= \dfrac{2 \pi}{T}$. In addition to rigorously considering the setting of real-valued pressure-volume fields, we also discuss multiharmonic algorithms derived from complex-valued fields (that is, from dropping the $\real$ operator in \eqref{cut off approx}). This approach is sometimes adopted in electromagnetism (see, e.g.,~\cite{de2001strong, yamada1988harmonic}) and it has been discussed in~\cite{kaltenbacher2021periodic} for the de-coupled Westervelt equation. Here, it results in simplified algorithms, which are, however, only formally investigated; see Section~\ref{sec: multiharmonic complex} for details. Toward reaching the overall goal of the work, our main contributions pertain to
\begin{itemize}[leftmargin=7mm]
	\setlength\itemsep{1mm}
	\item rigorously establishing the existence of time-periodic solutions for the Westervelt-ODE system in \eqref{West_RPE};
	\item deriving cutoff multiharmonic approximations of the system under time-periodic excitation for computing $(\pN_m, \vN_m)$;
	\item developing and analyzing \emph{linearized} multiharmonic cut-off algorithms for computing $(\pN_m, \vN_m) $.
\end{itemize}
In particular, we characterize the error of linearized multiharmonic algorithms for real-valued fields in terms of the number of harmonics and a contribution due to the fixed-point iteration; see~Theorem~\ref{thm: convergence}. For convenience, an overview of algorithms investigated in this work is provided in Figure \ref{fig: overview_alg}. \vspace*{2mm}

\definecolor{whispermint}{RGB}{240,250,245}
\begin{figure}[h!] 
	\centering
	\begin{tikzpicture}[scale = 0.95,
		algo/.style={transform shape, draw, rectangle, minimum height=1.2cm, minimum width=2.5cm, align=center},
		labelstyle/.style={font=\bfseries},
		node distance=1.2cm and 1.5cm
		]
		
		% Centered shared algorithm boxes
		\node[algo, fill=white!90!whispermint!10] (cutoff) {Multiharmonic\\cut-off algorithms};
		\node[algo, below= 1cm of cutoff, fill = white!90!whispermint!10] (fixedpoint) {Linearized multiharmonic \\ cut-off algorithms};
		
		% Complex-only algorithm
		\node[algo, right=1.5cm of cutoff, fill = white!90!whispermint!10] (twoharmonic) {Two-harmonic\\algorithm};
		
		% Define the center x position
		\coordinate (center) at ($(cutoff)!0.5!(twoharmonic)$);
		
		% === Draw background boxes manually ===
		\begin{scope}[on background layer]
			% Real-valued half (left side)
			\path let \p1 = (cutoff.north west), \p2 = (fixedpoint.south west) in
			node[fill=blue!5, draw=blue, dashed, thick, rounded corners, minimum width=6.2cm, minimum height=5.3cm, anchor=north east]
			at ($(cutoff.north)!0.5!(fixedpoint.south) + (-0.2, 3.4)$) {};
			
			% Complex-valued half (right side)
			\path let \p1 = (cutoff.north east), \p2 = (twoharmonic.south east) in
			node[fill=green!10, draw=green!50!black, dashed, thick, rounded corners, minimum width=6.2cm, minimum height=5.3cm, anchor=north west]
			at ($(cutoff.north)!0.5!(fixedpoint.south) + (0.2, 3.4)$) {};
		\end{scope}
		
		% Labels inside boxes
		\node[labelstyle, anchor=north] at ($(cutoff.north) + (-5.25, 1.6)$) {Real-valued};
		\node[labelstyle, anchor=north] at ($(cutoff.north) + (-4.3, 1.15)$) {pressure-volume field};
		\node[labelstyle, anchor=north] at ($(cutoff.north) + (5, 1.6)$) {Complex-valued};
		\node[labelstyle, anchor=north] at ($(cutoff.north) + (4.45, 1.15)$) {pressure-volume field};
		\node[labelstyle, font=\normalsize, anchor=north] at ($(cutoff.north) + (5.05, -3.3)$) {(simplified setting)};
		\draw[-{Stealth}, line width=1pt] (cutoff.east) -- (twoharmonic.west);
		\draw[-{Stealth}, line width=1pt] (cutoff.south) -- (fixedpoint.north);
	\end{tikzpicture}
	\caption{Overview of multiharmonic algorithms in this work.}
	\label{fig: overview_alg}
\end{figure}
\subsection{Novelty and related work} 
To the best of our knowledge, this is the first work to rigorously develop multiharmonic approaches for coupled nonlinear wave-ODE systems of this type, providing both well-posedness theory and numerical analysis. An approach to second harmonic generation for a linearization of \eqref{West_RPE} with $\eta =0$ and $b=0$ in the wave equation has been  formally set up in \cite[Ch.\ 5]{hamilton1998nonlinear} without convergence guarantees.  \\
\indent Rigorous multiharmonic studies have been performed on single-physics equations. Existence of periodic solutions of nonlinear acoustic models, including the Westervelt equation, has been investigated rigorously in~\cite{kaltenbacher2021periodic, Rainer2024nonlinear, kaltenbacher2024well, kaltenbacher2025acoustic, celik2018nonlinear, celik2019nonlinear}. An iterative multiharmonic algorithm for the Westervelt equation has been proposed and rigorously studied in~\cite{Rainer2024nonlinear}. A numerical algorithm based on using the complex Fourier Ansatz for the de-coupled Westervelt equation and  a boundary element approach for the resulting Helmholtz problems has been developed and investigated in \cite{groth2021accelerating}. In the context of electromagnetism, a multiharmonic treatment of the quasi-stationary Maxwell problem has been developed and rigorously analyzed~in~\cite{bachinger2005numerical}; numerical simulation aspects are discussed in~\cite{bachinger2006efficient}.  These results, however, do not address the coupling to microbubble/ODE dynamics. \\
\indent Periodic solutions of Rayleigh--Plesset-type equations that do not incorporate modeling of the acoustic propagation have been rigorously investigated in the literature; we refer to, for example,~\cite{yu2024bifurcation, hakl2013periodic} and~\cite[Ch.\ 9]{torres2015mathematical}, and the references provided therein.  We mention in passing that the mathematical literature on non-periodic models in nonlinear acoustics of non-bubbly media is quite rich; see, e.g.,~\cite{kaltenbacher2009global, kaltenbacher2015mathematics, acosta2022nonlinear, eptaminitakis2024weakly} and the references provided therein.
\subsection{Organization of the paper} The remainder of the exposition is organized as follows. In Section~\ref{sec: existence}, we determine sufficient conditions for the well-posedness of the time-periodic boundary value problem for the Westervelt-ODE system \eqref{West_RPE}. In Section \ref{sec: multiharmonic}, we derive a multiharmonic cut-off representation of the system and propose a linearization which results in a simplified setting for computing Fourier coefficients in \eqref{cut off approx}. In Section~\ref{sec: convergence}, we characterize the error and prove the convergence of the proposed linearized multiharmonic scheme to the solution of the Westervelt-ODE system as the number of harmonics $N$ tends to $\infty$. The main theoretical result of the section is contained in Theorem~\ref{thm: convergence}. In Section \ref{sec: multiharmonic complex}, we discuss multiharmonic algorithms resulting from complex solutions fields. Finally, in Section \ref{sec:numerics}, we investigate and compare the introduced algorithms.

\section{Existence of time-periodic solutions}  \label{sec: existence}
In this section, we establish the basis for the numerical analysis by investigating the existence and uniqueness of periodic solutions for the coupled PDE-ODE system. More precisely, given $T>0$ and a bounded domain $\Omega \subset \R^d$, where $d \in \{2,3\}$, we study the  problem:
\begin{equation}	\label{ibvp:westervelt volume periodic}
	\left\{	\begin{aligned}
	\ &\ptt - c^2 \Delta p - b \Delta \pt = \eta (p^2)_{tt}+c^2 \rho_0 n_0(x) \vtt + h \quad && \text{in } \Omega \times  (0,T),\\
&\beta \pt + \gamma p + \nabla p \cdot \vecn = 0 && \text{on }  \partial\Omega \times (0,T), \\
&p(0) = p(T), \ \pt(0) = \pt(T) && \text{in } \Omega, \\ 
& \vtt + \delta \omega_0 \vt + \omega_0^2 v = \ca v^2 + \cb(2v \vtt + \vt^2)- \mu p \quad && \text{in }   \Omega \times (0,T),\\
&v(0) = v(T), \ \vt(0) = \vt(T) && \text{in } \Omega.
\end{aligned} \right.
\end{equation}
The analysis will be based on successive approximations of the system, for which knowledge of linear time-periodic ODE and wave problems will be very helpful. We thus discuss those results next. \\[1mm]
\noindent{\bf Notation}. Below we use $\textup{lhs} \lesssim \textup{rhs}$ to denote $\textup{lhs} \leq C \cdot \textup{rhs}$, where $C>0$ is a generic constant. When writing norms in Bochner spaces, we omit the temporal domain $(0,T)$. For example, $\|\cdot\|_{L^p(L^q(\Omega))}$ denotes the norm in $L^p(0,T; L^q(\Om))$.
\subsection{Auxiliary existence results for linear time-periodic problems} 
To set up the local well-posedness analysis of \eqref{ibvp:westervelt volume periodic}, we first state two separate results on the well-posedness of an ODE (describing damped oscillations with forcing) and a linear time-periodic wave problem.
\begin{lemma} \label{lemma: ode}
	Let  $T>0$ and $f \in L^2(0,T; \Linf)$. Furthermore, let $\delta$, $\omega_0>0$. Then, the periodic ODE problem
	\begin{equation} \label{ip v lin}
		\left\{	\begin{aligned}
			& \vtt + \delta \omega_0 \vt + \omega_0^2 v = f  && \text{a.e.\ in }   \Omega \times (0,T),\\
			&v(0) = v(T), \ \vt(0) = \vt(T) && \text{a.e.\ in } \Omega,
		\end{aligned} \right.
	\end{equation}
	has a unique solution 
		\begin{equation} \label{def Xv}
			\begin{aligned}
			v \in	\Xv=\{v \in H^2(0,T; \Linf): \, v(0)=v(T), \, v_t(0)= v_t(T) \ \text{a.e.}\}
			\end{aligned}
		\end{equation}
	 that satisfies
	\begin{equation}
		\begin{aligned}
		\|v\|_{\Xv} \coloneqq	\|v\|_{L^\infty(\Linf)}+	\|\vt\|_{L^\infty(\Linf)}+\|\vtt\|_{L^2(\Linf)} \lesssim \|f\|_{L^2(\Linf)}.
		\end{aligned}
	\end{equation}	
\end{lemma}
\begin{proof}
The existence and uniqueness can be established using Floquet theory (see, for example~\cite{schmidt1976yakubovich} and~\cite[Ch.\ 3]{grimshaw2017nonlinear}) as the only solution of the homogeneous problem is zero due to the fact that $\delta \omega_0>0$. The details are provided in Appendix~\ref{proof Lemma} for completeness.
\end{proof}
%We next recall a result from~\cite{kaltenbacher2021periodic} on a linear time-periodic wave problem. 
\begin{proposition}[see{~\cite[Theorem 2.1]{Rainer2024nonlinear}}]\label{prop: periodic lin wave}
Let $T>0$, $\Omega \subset \mathbb{R}^d$, where $d \in \{2,3\}$, be a bounded domain with $C^{1,1}$ boundary, and $\beta,\,\gamma>0$, $c$, $b>0$.  Let $h\in \LtwoTLtwo$. Then there exists a unique (weak) solution 
\begin{equation} \label{def Xp}
\begin{aligned}	
	p \in \Xp =\, \begin{multlined}[t] \Bigl\{p \in H^2(0,T; \Ltwo)  \cap H^1(0,T;H^{3/2}(\Omega)) \cap L^2(0,T;H^2(\Omega)):\\   
		|| \trace{\Omega}{\nabla p\cdot\vecn} ||_{H^1(L^2(\partial\Omega))} < \infty,\ \|\Delta \pt\|_{L^2(\Ltwo)}<\infty,\, p(0) = p(T),\ \pt(0) = \pt(T) \text{ a.e.}\Bigr\}
		\end{multlined}
\end{aligned}	
\end{equation} 
of the time-periodic boundary-value problem
\begin{equation}\label{eq:wave:linear:periodic}
\left\{	\begin{aligned}
			&\ptt - c^2 \Delta p - b \Delta \pt = h\quad && \text{in }  \Omega \times  (0,T),\\
			&\beta \pt + \gamma p + \nabla p \cdot \vecn = 0 && \text{on } {\partial\Omega \times (0,T)}, \\
			&p(0) = p(T), \ \pt(0) = \pt(T) && \text{in } \Omega. 
		\end{aligned}   \right.
\end{equation}
The solution satisfies
\begin{equation} \label{est ln wave}
\begin{aligned}
	\Tnorm{p}\coloneqq&\,\begin{multlined}[t] ||p||_{\LtwoHtwo} +	||\pt||_{\LtwoHthreehalf}+\|\ptt\|_{\LtwoLtwo} 
		+ \|\Delta \pt\|_{L^2(\Ltwo)} + \|\nabla p \cdot \bfn\|_{H^1(L^2(\partial \Omega))} 
		\end{multlined}\\
	\leq&\,{ C(b, \gamma, c, \beta, T, \Omega) ||h||_{ L^2(0,T;L^2(\Omega))}},
\end{aligned}	
\end{equation}
for some constant $C=C(b, \gamma,c,\beta, T, \Omega) > 0$, independent of $h$.
\end{proposition}
We note that having $\gamma>0$ in \eqref{eq:wave:linear:periodic} is crucial for uniqueness. Furthermore, having strong damping (that is, $-b \Delta \pt$ with $b>0$) in the wave equation contributes to its parabolic-like character, and is heavily exploited in the proof of the above statement. We also note that the functions $f$ and $h$ do not have to be time periodic for the well-posedness results in this section, although we will make that assumption on the acoustic source in Section~\ref{sec: multiharmonic} to develop (iterative) multiharmonic algorithms.
\subsection{Analysis of the Westervelt-ODE system}  
Equipped with the previous results on linear problems, we are now ready to prove the well-posedness of the coupled pressure-volume problem. To state it, we introduce 
	\begin{equation} \label{def ballr}
	\begin{aligned}
		\ballr = \left\{ (p, v) \in \Xp \times \Xv:\, \|p\|_{\Xp} \leq \rp, \ \|v\|_{\Xv} \leq \rv \right \},
	\end{aligned}
\end{equation}
where $\rp>0$ and $\rv>0$ will be set as small enough in the course of the proof below. The proof will have constructive nature using successive approximations. 
 \begin{theorem} \label{thm: existence West-volume}
Let $T>0$, $\Omega \subseteq \mathbb{R}^d$, $d \in \{2,3\}$, be open, bounded, connected, with $C^{1,1}$ boundary, $\beta$, $\gamma$,  $c$,  $b$,  {$\delta$, $\omega_0>0$} and $\eta$, $\mu$, $\ca$, $\cb \in \R$, $n_0 \in L^\infty(\Omega)$, $h \in L^2(0,T; \Ltwo)$. 
Then there exist $\deltap>0$ and $\deltar>0$, such that if 
\begin{equation} \label{smallness condition}
	||h||_{\LtwoTLtwo} \leq \deltap \quad \text{and} \quad \|\change{n_0}\|_{\Linf} \leq \deltar, 
\end{equation}
then there exist radii $\rp>0$ and $\rv>0$, such that problem \eqref{ibvp:westervelt volume periodic} has a unique solution $(p, v) \in \ballr$.
\end{theorem}
\begin{proof}
	As announced, the proof follows by constructing the solutions using successive approximations. Let $\rp$, $\rv>0$. We take $(\pfixedzero, \vfixedzero) \in \ballr$ and set up successive approximations using the system given by
	\begin{subequations}\label{iterative method proof}
	\begin{equation}	\label{eq pN proof}
		\left\{	\begin{aligned}
			\ &\pfixedN_{tt} - c^2 \Delta \pfixedN - b \Delta \pfixedN_t = \eta ((\pfixedNnegone)^2)_{tt}+c^2 \rho_0 n_0 \vfixedN_{tt} + h \, && \text{in } \Omega \times  (0,T),\\
			&\beta \pfixedN_t + \gamma \pfixedN + \nabla \pfixedN \cdot \vecn = 0 && \text{on }  \partial\Omega \times (0,T), \hspace{-0.28cm}\\
			&\pfixedN(0) = \pfixedN(T), \ \pfixedN_t(0) = \pfixedN_t(T) && \text{in } \Omega,
		\end{aligned} \right.
	\end{equation}
	and
	\begin{equation} \label{eq vN proof}
		\left\{	\begin{aligned}
			&	\vfixedN_{tt} + \delta \omega_0 \vfixedN_t + \omega_0^2 \vfixedN && \\
			=&\,  -  \mu\pfixedNnegone +\ca (\vfixedNnegone)^2  - \cb \left( 2 \vfixedNnegone \vfixedNnegone_{tt} + (\vfixedNnegone_t)^2 \right)  && \text{in }   \Omega \times (0,T),\\
			&\vfixedN(0) = \vfixedN(T), \ \vfixedN_t(0) = \vfixedN_t(T) && \text{in } \Omega.
		\end{aligned} \right.
	\end{equation}
\end{subequations}
We proceed through several steps. \\[1mm]
\noindent (i)	Let $N=1$. 	Note that we can first solve \eqref{eq vN proof} and use its solution as the input for solving \eqref{eq pN proof}. By Lemma~\ref{lemma: ode}, we have a unique $\vfixedone \in \Xv$ that solves \eqref{eq vN proof}, such that
	\begin{equation} \label{est vN one}
		\begin{aligned}
			\|\vfixedone\|_{\Xv} \lesssim  \|\pfixedzero\|_{L^2(\Linf)}+ \|\vfixedzero\|^2_{\Xv}.
		\end{aligned}
	\end{equation}
	From here, we find that
	\begin{equation} 
		\begin{aligned}
			\|\vfixedone\|_{\Xv} \lesssim \rp+ \rv^2,
		\end{aligned}
	\end{equation}
	and we can conclude that $\|\vfixedone\|_{\Xv} \leq \rv$ provided \change{$\rp$ and $\rv$ are small enough}. By Proposition~\ref{prop: periodic lin wave}, the linear problem in \eqref{eq pN proof} for $N=1$ has a unique solution $\pfixedone \in \Xp$, such that
	\begin{equation} \label{est pN one}
	\begin{aligned}
		\|\pfixedone\|_{\Xp} 
		\lesssim& \,  \| -\eta((\pfixedzero)^2)_{tt}  - c^2 \rho_0 n_0  {\vfixedone_{tt}}  + h\|_{\LtwoLtwo} \\
		\lesssim& \,  \|\pfixedzero\|^2_{\Xp} + \|n_0\|_{\Linf}\|\vfixedone\|_{\Xv}  + \|h\|_{\LtwoLtwo}.
	\end{aligned}
\end{equation}	
By using the fact that $\|\vone\|_{\Xv} \leq \rv$, we obtain
	\begin{equation}
	\begin{aligned}
		\|\pfixedone\|_{\Xp} 
		\lesssim \rp^2 + \|n_0\|_{\Linf} \rv  + \|h\|_{\LtwoLtwo}.
	\end{aligned}
\end{equation}	
From here, we can conclude that $\|\pfixedone\|_{\Xp} \leq \rp$ for small enough $\|h\|_{L^2(\Ltwo)}$ by reducing $\rp$ and $\|n_0\|_{\Linf}$. \\[2mm]
\noindent (ii)	  Let now $(\pfixedNnegone, \vfixedNnegone) \in \ballr$. By Lemma~\ref{lemma: ode}, we have
	\begin{equation} \label{est vN}
		\begin{aligned}
			\|\vfixedN\|_{\Xv} \lesssim \|\pfixedNnegone\|_{L^2(\Linf)}+ \|\vfixedNnegone\|^2_{\Xv}
		\end{aligned}
	\end{equation} 
	and by Proposition~\ref{prop: periodic lin wave}, the linear problem in \eqref{eq pN proof} with periodic time conditions has a unique solution $\pfixedN \in \Xp$, such that
	\begin{equation} \label{est pN}
		\begin{aligned}
			\|\pfixedN\|_{\Xp} 
			\leq& \, C \| -\eta((\pfixedNnegone)^2)_{tt}  - c^2 \rho_0 n_0  \vfixedN_{tt}  + h\|_{\LtwoLtwo}\\
			 \lesssim&\,  \|\pfixedNnegone\|^2_{\Xp} + \|n_0\|_{\Linf}\|\vfixedN\|_{\Xv}  + \|h\|_{\LtwoLtwo}.
		\end{aligned}
	\end{equation}	
	By then proceeding analogously to the case $N=1$, we obtain
	\begin{equation} \label{bounds}
		\begin{aligned}
		\|\pfixedN\|_{\Xp} \leq \rp, \quad \|\vfixedN\|_{\Xv} \leq \rv,
		\end{aligned}
	\end{equation}
provided $\|n_0\|_{\Linf}$ and $\|h\|_{L^2(\Ltwo)}$ are small enough. 
We thus conclude that for small enough $\|h\|_{\LtwoLtwo}$ and $\|n_0\|_{\Linf}$, there exist $\rp>0$ and $\rv>0$, such that if $(\pfixedzero, \vfixedzero) \in \ballr$, then for each $N \geq 1$, problem \eqref{iterative method proof} has a unique solution $(\pfixedN, \vfixedN) \in \ballr$.	\\[2mm]
\noindent (iii)	 We next wish to show that $\{(\pfixedN, \vfixedN)\}_{N \geq 1}$ is a Cauchy sequence in $\Xp \times \Xv$. 
 We note that for $N_1$, $N_2 \in \N$, the difference $(\op, \ov)=(\pfixedNone-\pfixedNtwo, \vfixedNone-\vfixedNtwo)$ solves
	\begin{subequations}\label{Cauchy proof}
	\begin{equation}\label{weak eq pN Cauchy}
	\left\{	\begin{aligned}
		\ &\op_{tt} - c^2 \Delta \op - b \Delta \op_t && \\
		=&\, -\eta((\pfixedNonenegone-\pfixedNtwonegone)(\pfixedNonenegone+\pfixedNtwonegone))_{tt}  - c^2 \rho_0 n_0  (\vfixedNone_{tt}-\vfixedNtwo_{tt}) \, && \text{in } \Omega,\\
		&\beta \op_t + \gamma \op + \nabla \op \cdot \vecn = 0 && \text{on }  \partial\Omega, \hspace{-0.28cm}\\
	\end{aligned} \right.
\end{equation}
	 and 
	\begin{equation} \label{weak eq vN Cauchy}
		\left\{		\begin{aligned}
		&\ov_{tt} + \delta \omega_0 \ov_t + \omega_0^2 \ov \\
		=&\,\begin{multlined}[t]
			- \mu (\pfixedNonenegone-\pfixedNtwonegone)+  \ca (\vfixedNonenegone-\vfixedNtwonegone)(\vfixedNonenegone+\vfixedNtwonegone) 
			\\- \cb \bigl( 2 (\vfixedNonenegone-\vfixedNtwonegone) \vfixedNonenegone_{tt}\\ +\vfixedNtwonegone(\vfixedNonenegone_{tt}-\vfixedNtwonegone_{tt}) 
			+ (\vfixedNonenegone_t-\vfixedNtwonegone_t)(\vfixedNonenegone_t+\vfixedNtwonegone_t)\bigr) \quad  \text{in } \Omega.
		\end{multlined}
		\end{aligned} \right.
	\end{equation}
\end{subequations}
From here, similarly to before, using Proposition~\ref{prop: periodic lin wave} and Lemma~\ref{lemma: ode}, we obtain
\begin{equation} \label{10}
	\begin{aligned}
		\|\pfixedNone-\pfixedNtwo\|_{\Xp} \lesssim \rp \|\pfixedNonenegone-\pfixedNtwonegone\|_{\Xp} + \|n_0\|_{\Linf}\|\vfixedNone-\vfixedNtwo\|_{\Xv}
	\end{aligned}
\end{equation}
and
\begin{equation} \label{1a}
	\begin{aligned}
		\|\vfixedNone-\vfixedNtwo\|_{\Xv} \lesssim&\, \begin{multlined}[t]  \|\pfixedNonenegone-\pfixedNtwonegone\|_{\Xp}+ \rv \|\vfixedNonenegone-\vfixedNtwonegone\|_{\Xv}.
		\end{multlined}
	\end{aligned}
\end{equation}
By adding $\eqref{10}+ \lambda \cdot \eqref{1a}$, with $\lambda>0$, we obtain
\begin{equation}
	\begin{aligned}
		&\|\pfixedNone-\pfixedNtwo\|_{\Xp} + (\lambda - C \|n_0\|_{\Linf})\|\vfixedNone-\vfixedNtwo\|_{\Xv}\\
		 \leq&\,
C	(	\lambda + \rp) \|\pfixedNonenegone-\pfixedNtwonegone\|_{\Xp}  + C \lambda \rv  \|\vfixedNonenegone-\vfixedNtwonegone\|_{\Xv}
	\end{aligned}
\end{equation}
for some $C>0$. 
We then impose $C \|n_0\|_{\Linf} < \lambda$,
and choose $\rp$, $\lambda$, and $\rv$ sufficiently small to conclude that $\{(\pfixedN, \vfixedN)\}_{N \geq 1}$ is a Cauchy sequence. Thus, there exists $(p,v)$ such that $\{(\pfixedN, \vfixedN)\}_{N \geq 1}$ converges to it in $\Xp \times \Xv$. Passing to the limit in \eqref{iterative method proof} proves that $(p,v)$ solves the pressure-volume problem. \\[2mm]
\noindent (iv)	Finally, to show that such a constructed solution is unique, we take 
\[
(\puniqone, \vuniqone), \, (\puniqtwo, \vuniqtwo) \in \ballr,
\]
and note that the difference $(\op, \ov)=(\puniqone-\puniqtwo, \vuniqone-\vuniqtwo)$ solves
	\begin{subequations}\label{uniqueness}
		\begin{equation}\label{weak eq p uniqueness}
		\left\{	\begin{aligned}
			\ &\op_{tt} - c^2 \Delta \op - b \Delta \op_t =  -\eta(\op(\puniqone+\puniqtwo))_{tt}  - c^2 \rho_0 n_0  \ov_{tt} \, && \text{in } \Omega \times  (0,T),\\
			&\beta \op_t + \gamma \op + \nabla \op \cdot \vecn = 0 && \text{on }  \partial\Omega \times (0,T), \hspace{-0.28cm}\\
		\end{aligned} \right.
	\end{equation}
	 and 
	\begin{equation} \label{weak eq v uniqueness}
		\begin{aligned}
			\ov_{tt} + \delta \omega_0 \ov_t + \omega_0^2 \ov 
			=- \mu \op +  \ca \ov \left(\vuniqone+\vuniqtwo\right) - \cb \left( 2 (\ov \vuniqone_{tt}+\vuniqtwo\ov_{tt} )
				+ \ov_t(\vuniqone_t+\vuniqtwo_t)\right).
		\end{aligned}
	\end{equation}
\end{subequations}	
We obtain analogously to before
	\begin{equation} \label{11}
		\begin{aligned}
			\|\op\|_{\Xp} \lesssim \rp \|\op\|_{\Xp}+ \|n_0\|_{\Linf} \|\ov\|_{\Xv}
		\end{aligned}
	\end{equation}
	and
	\begin{equation} \label{22}
		\begin{aligned}
			\|\ov\|_{\Xv} \lesssim \|\op\|_{\Xp}+ \rv \|\ov\|_{\Xv}.
		\end{aligned}
	\end{equation}
For $\rv$, $\rp$, and $\|n_0\|_{\Linf}$ small enough, these inequalities allow us to conclude that $\|\op\|_{\Xp}=\|\ov\|_{\Xv}=0$. This step completes the proof. \qed
\end{proof}

The small-solution setting of Theorem~\ref{thm: existence West-volume} effectively forces the nonlinear ODE to retain the damped harmonic oscillator structure with $T$-periodic coefficients. Indeed, under the assumptions of Theorem~\ref{thm: existence West-volume}, for $\rv$ small enough, so that
	\begin{equation}
		\max\, \left\{ 2| \cb|, \frac{1}{\delta \omega_0}| \cb|, \frac{1}{\omega_0^2} |\ca| \right\} \cdot \rv <1,
	\end{equation}
	 we can rewrite the ODE in the following form:
\begin{equation} \label{OD rewritten}
	\begin{aligned}
		\vtt + \frac{\delta \omega_0-\cb \vt}{1-2 \cb v} \vt + \frac{\omega_0^2-\ca v}{1-2 \cb v} v =- \frac{\mu}{1-2\cb v} p.
	\end{aligned}
\end{equation}
Setting $\ell = \dfrac{\delta \omega_0-\cb \vt}{1-2 \cb v}>0$, $q =\dfrac{\omega_0^2-\ca v}{1-2 \cb v}>0$, and $f = - \dfrac{\mu}{1-2\cb v} p$,  we see that the volume fluctuation $v$ can be represented as the solution of a second-order ODE (pointwise a.e.\ in space) with positive $T$-periodic coefficients and a $T$-periodic right-hand side, namely
\begin{equation} 
		\begin{aligned}
		&  \vtt+\ell(t)\vt+q(t) v =f(t),
	\end{aligned}
\end{equation}
where $\ell(0)=\ell(T)$, $q(0)=q(T)$,  $f(0)=f(T)$. 
\subsection{On the smallness assumption} \label{sec: smallness}
The assumption \eqref{smallness condition} on the smallness of the microbubble number density  $n_0$ in Theorem~\ref{thm: existence West-volume} regulates the strength of the source of acoustic waves, along with the smallness of the external source $h$.  To assess how meaningful this assumption is, we consider a nondimensional system obtained by introducing the scaling
\begin{equation}
	\begin{aligned}
		x = L \tilde{x}, \qquad t = \frac{L}{c} \tilde{t}, \quad p = \pref \, \tilde{p},  \qquad v = \vref \, \tilde{v},
	\end{aligned}
\end{equation}
where $L$ is a characteristic length scale, $\pref$ is a reference pressure amplitude, and $\vref$ a reference volume.  After the change of variables and division by $\frac{c^2 \pref}{L^2}$, we obtain the dimensionless pressure equation:
	\begin{equation} \label{West_RPE nondim}
			\begin{aligned}
				&\ \ \tilde{p}_{tt} -  \Delta \tilde{p} - \tilde{b} \Delta \tilde{p}_t = \tilde{\eta} (\tilde{p}^2)_{tt} + \tilde{\kappa}  {n}_0 \tilde{v}_{tt} 
				+\tilde{ \alpha} h \quad &&\text{in } \Omega \times (0,T), 
			\end{aligned} 
	\end{equation}
	where the relevant transformed coefficient is given by
	\begin{equation}
\tilde{\kappa} {n}_0 = \frac{ \rho_0 c^2 \vref }{\pref} {n}_0.
	\end{equation}
Since $\vref n_0$ can be understood as the gas volume fraction (see~\cite[Sec.~2]{nikolic2024mathematicalmodelsnonlinearultrasound}), the smallness assumption corresponds to requiring moderate concentrations of microbubbles. For typical parameters in medical imaging, $c \sim 1500$m$/$s, $\rho \sim 1000\, $kg$/$m$^3$, $\pref \in [10^5 , 10^6]\, $Pa, and $R_0 \sim 10^{-6}\,$m, $\vref = \frac43 R_0^3 \pi$, we obtain $\tilde{\kappa} n_0 \ll 1$ for microbubble densities in the range $n_0 \in [10^9, 10^{12}]\,$ bubbles$/m^3$.

\section{Time discretization via a multiharmonic Ansatz for real fields} \label{sec: multiharmonic}
In this section, we discuss the time discretization of the system by using a multiharmonic Ansatz, in the general spirit of~\cite{kaltenbacher2021periodic}. We focus on a special case of having $T$-periodic acoustic excitation. That is, we assume that  the acoustic source term has the form $h=g_{tt}$, where the function $g$ is given by 
\begin{equation} \label{def g}
g(x,t)=\real \left\lbrace \sum_{m=0}^M \exp(\imath m \omega t)h_m^M(x)\right\rbrace, \quad \omega = \dfrac{2 \pi}{T}, \quad h_m^M \in L^2(\Omega; \C).
\end{equation}
Under the assumptions of Theorem~\ref{thm: existence West-volume}, a unique time-periodic solution $(p, v)$ of the system given in \eqref{ibvp:westervelt volume periodic} exists. 
For numerical approximations,  we thus employ the following multiharmonic Ansatz: 
	\begin{equation} \label{finite_approx}
			\begin{aligned}
					u^N(x,t)=&\, \frac{1}{2} \sum_{m=0}^{N} \left( \exp(\imath m \omega t) u_m^N(x) + \exp(-\imath m \omega t) \overline{u^N_m(x)} \right) \\
					=&\, \real \left\lbrace \sum_{m=0}^{N}  \exp(\imath m \omega t) u_m^N(x)\right\rbrace.
				\end{aligned}
		\end{equation}
	Alternatively, the multiharmonic Ansatz can be written in a more standard Fourier series form:
	\begin{equation}
		u^N(x,t) = \sum_{m=0}^N \left [u^c_m(x) \cos(m \omega t)+ u^s_m(x)\sin(m\omega t)\right ]
	\end{equation}
	with $u^N_m(x) = u_m^c(x)- \imath u_m^s(x)$. We then look for the approximate pressure-volume field $(\pN, \vN) \in \XN \times \XN$, with
\begin{equation}
	X_N = \left\lbrace \real \left\lbrace \sum_{m=0}^N \exp( \imath m \omega t) u_m^N(x) \right\rbrace \, : \, u_m^N \in H^2(\Omega; \C) \right\rbrace,
\end{equation}
where the coefficients $(\pN_m, \vN_m)$ in the expansion are determined from the following system: 
\begin{equation} 	\label{semi-discrete}
		\left\{	\begin{aligned}
			\ &\pttN - c^2 \Delta \pN - b \Delta \ptN &&\\ 
			=&\, \eta\,\projN \left[ ((\pN)^2)_{tt}\right]
			+c^2 \rho_0 n_0 \vN + \projN h  &&\text{in } \Omega \times  (0,T),\\[1mm]
			&\beta \ptN + \gamma \pN + \nabla \pN \cdot \vecn = 0 &&\ \text{on }  \partial\Omega \times (0,T),\\[1mm]
			&	\vttN+ \delta \omega_0 \vtN + \omega_0^2 \vN &&\\
			=&\,  -  \mu \pN +  \,\projN \left[\ca (\vN)^2 
		+ \cb \left( 2 \vN \vttN + (\vtN)^2 \right) \right]  &&\text{in }   \Omega \times (0,T).
		\end{aligned} \right.
\end{equation}
We next show that \eqref{semi-discrete} can be equivalently rewritten as a system of Helmholtz problems and algebraic equations for computing $(\pN_m, \vN_m)$.
\subsection{A multiharmonic cut-off algorithm}
 Below we skip writing the dependencies of Fourier coefficients on $x$ for readability. 

\begin{proposition}\label{prop: multi} 
	Let the assumptions of Theorem~\ref{thm: existence West-volume} hold with the acoustic source term assumed to have the form $h=g_{tt}$, where $g$ is given in \eqref{def g}.  The problem in 	\eqref{semi-discrete} yields the following coupled system for computing $(\pN_m, \vN_m)$
for $N \geq 0$ and $0 \leq m \leq N$: \small
\begin{align}
	%	\begin{aligned}
		m = 0:& \begin{cases}
			\begin{aligned}
				\text{(i)} & \quad p_0^N =0, \\
				\text{(ii)} & \quad \omega_0^2 v_0^N = \ca \left( v_0^N \right)^2 + \sum_{j=1}^N \left( \frac{\ca}{2} - \frac{\cb}{2} \omega^2 j^2 \right) \left| v_j^N \right|^2, \\
			\end{aligned} 
		\end{cases} \\[0.5em]
		m= 1:&  \begin{cases}
			\begin{aligned}
				\text{(i)} & \quad -  p_1^N - \frac{c^2 + \imath b \omega }{\omega^2} \Delta p_1^N+ c^2 n_0\rho_0  v_1^N = - h_1^N - \sum_{k=3:2}^{2N-1}  \eta \overline{p_{\frac{k-1}{2}}^N} p_{\frac{k+1}{2}}^N,  \\[0.3em]
				\text{(ii)} & \quad \left( - \omega^2 + \imath \delta \omega_0 \omega + \omega_0^2  \right) v_1^N + \mu p_1^N\\
				& \quad = (\ca -\cb \omega^2) v_0^N v_1^N + \sum_{k=1:2}^{2N-1} \left( \ca - \cb \omega ^2 \frac{k^2 + 3}{4}\right) \overline{v_{\frac{k-1}{2}}^N} v_{\frac{k+1}{2}}^N,  \\
			\end{aligned}
		\end{cases}  \\[0.5em]
		m = \left\lbrace 2, \ldots, N \right\rbrace:& \begin{cases}
			\begin{aligned}
				\text{(i)} & \quad - p_m^N- \frac{ c^2 + \imath m b \omega  }{m^2 \omega^2} \Delta p_m^N + c^2 \rho_0 n_0 v_m^N\\
				& \quad = - h_m^N -\eta \left(  \sum_{l=1}^{m-1}  p_l^N p_{m-l}^N- 2 \sum_{k=m+2:2}^{2N-m}  \overline{p_{\frac{k-m}{2}}^N} p_{\frac{k+m}{2}}^N  \right), \\[0.5em]
				\text{(ii)}&  \quad  \left(- \omega^2m^2 + \imath \delta \omega_0 \omega m+\omega_0^2\right)   v_m^N + \mu p_m^N, \\
				& \quad = \sum_{l=0}^m \left( \frac{\ca}{2} - \frac{\cb \omega^2}{2} (m-l)(2m-l) \right)  v_l^N v_{m-l}^N \\
				& \quad \quad + \sum_{k=m:2}^{2N-m} \left( \ca - \cb \omega ^2 \frac{k^2 + 3m^2}{4}\right) \overline{v_{\frac{k-m}{2}}^N} v_{\frac{k+m}{2}}^N. \\
			\end{aligned}
		\end{cases}
		%	\end{aligned}
\end{align}
\normalsize
Additionally, each of the individual functions $p_m^N$ for $m \in \left\lbrace 1, \ldots, N \right\rbrace$ satisfies the following boundary conditions:
\begin{equation}\label{eq: bc}
	(\imath \omega m \beta + \gamma ) p_m^N + \nabla p_m^N \cdot \vecn = 0 \qquad \text{ on }\, \partial \Omega.
\end{equation}
\end{proposition}

\begin{proof} We begin by considering the derivation of the multiharmonic algorithm for the Westervelt's equation and the ODE separately. For the PDE, one can follow the steps  outlined in~\cite{kaltenbacher2021periodic}, where the de-coupled Westervelt equation is considered, by setting the right-hand side to $g = h + c^2 \rho_0 n_0 v$. Projections onto $\XN$ of product terms appearing on the right-hand side of \eqref{semi-discrete} can be expressed using~\cite[Lemma 3.1]{kaltenbacher2021periodic}.\\ 
\indent For the ODE in \eqref{semi-discrete}, the Ansatz given in \eqref{finite_approx}, together with \cite[Lemma 3.1]{kaltenbacher2021periodic}, yields 
\small
\begin{align}
0 = & \real \left\lbrace- \omega^2 \sum_{m=0}^N v_m^N m^2 \exp(\imath m \omega t) + \delta \omega_0 \imath \omega \sum_{m=0}^N v_m^Nm \exp( \imath m \omega t) + \omega_0^2 \sum_{m=0}^N  v_m^N \exp(\imath m \omega t) \right. \\
& \left. + \mu \sum_{m=0}^N p_m^N \exp(\imath m \omega t)- \frac{\ca}{2} \left( (v_0^N)^2 + \sum_{j=0}^N \left| v_j^N\right|^2 \right. \right. \\
& \left. \left. + \sum_{m=1}^N \left[ \sum_{l=0}^m v_l^N v_{m-l}^N + 2 \sum_{k=m:2}^{2N-m} \overline{v_{\frac{k-m}{2}}^N} v_{\frac{k+m}{2}}^N \right] \exp(\imath m \omega t) \right) + \cb \frac{\omega^2}{2}\left( \sum_{j=0}^N j^2 \left| v_j^N\right|^2 \right. \right. \\
& \left. \left. +  \sum_{m=1}^N \left[ \sum_{l=0}^m  (m-l)(2m-l)  v_l^N v_{m-l}^N + \sum_{k=m:2}^{2N-m}  \frac{k^2 +3m^2}{2}  \overline{v_{\frac{k-m}{2}}^N} v_{\frac{k+m}{2}}^N \right] \exp(\imath m \omega t) \right) \right\rbrace.
\end{align}
\normalsize
From here, we have
\small 
\begin{align}
	0 
	= & \real \left\lbrace \sum_{m=0}^N \left[ \left( - \omega^2  m^2 + \delta \omega_0 \imath \omega m + \omega_0^2 \right) v_m^N + \mu p_m^N \right] \exp(\imath m \omega t) \right. \\
	& \left. - \frac{\ca}{2} (v_0^N)^2 +   \sum_{j=0}^N \left( - \frac{\ca}{2} + \cb \frac{\omega^2}{2} j^2 \right)\left| v_j^N\right|^2  + \sum_{m=1}^N \left[  \sum_{l=0}^m \left( -\frac{\ca}{2} + \cb \frac{\omega^2}{2} (m-l)(2m-l) \right) v_l^N v_{m-l}^N \right. \right. \\
	& \left. \left. + \sum_{k=m:2}^{2N-m} \left( -\ca + \cb \omega^2 \frac{k^2 + 3m^2}{4}\right)  \overline{v_{\frac{k-m}{2}}^N} v_{\frac{k+m}{2}}^N \right]  \exp(\imath m \omega t) \right\rbrace. 
	%\end{aligned}
\end{align}
\normalsize
Combining our derivations above and using linear independence of the functions $ t \mapsto \exp( \imath m \omega t)$ yields a coupled nonlinear system for the functions $\left\lbrace p_0^N, v_0^N, \cdots, p_N^N, v_N^N \right\rbrace$. Additionally, each of the individual functions $p_m^N$ must satisfy the boundary conditions in \eqref{eq: bc}.  The 0-th equation is given by $\Delta p^N_0 = 0$, which together with the 0-th boundary condition (and the fact that $\gamma>0$)  implies that $p_0^N$ vanishes. By dividing the resulting PDEs by $m^2\omega^2$ for $m \geq 1$, we arrive at the claimed coupled system. \qed
\end{proof}

We observe that without the sums over $k$ on the right-hand side of the system in Proposition~\ref{prop: multi}, the system would be triangular and could be solved by substitution. One way of obtaining a triangular form is by linearizing the right-hand side as we show next. 
\subsection{A linearized multiharmonic cut-off algorithm}  \label{sec: linearized algorithm}
Next, we propose a linearized multilevel method. That is, for $N \geq 1$, we consider the sequence of the following linearized equations where the quadratic terms are approximated by taking into account $N-1$ harmonics:
\begin{equation} 	\label{system_lin}
	\left\{	\begin{aligned}
		\ &\pttN - c^2 \Delta \pN - b \Delta \ptN &&\\ 
		=&\, \eta\,\projN \left[ ((\pNone)^2)_{tt}\right]
		+c^2 \rho_0 n_0 \vttN + \projN h  &&\text{in } \Omega \times  (0,T),\\[1mm]
		&\beta \ptN + \gamma \pN + \nabla \pN \cdot \vecn = 0 &&\text{on }  \partial\Omega \times (0,T),\\[1mm]
		&	\vttN+ \delta \omega_0 \vtN + \omega_0^2 \vN &&\\
		=&\,  -  \mu \pN +  \,\projN \left[\ca (\vNone)^2 
		+ \cb \left( 2 \vNone \vttNone + (\vtNone)^2 \right) \right]  &&\text{in }   \Omega \times (0,T),
	\end{aligned} \right.
\end{equation}
where for $u^N \in X_N$, we formally define $u_m^N = 0$ for all $m>N$. To start the algorithm, we set $p^0 =0$ and $v^0 = 0$. We next show that this problem can be seen as a triangular system of Helmholtz problems and algebraic equations that can be solved by successive substitution. 
\begin{proposition} \label{prop: multilevel_linearized}
	Let the assumptions of Theorem~\ref{thm: existence West-volume} hold with the acoustic source term assumed to have the form $h=g_{tt}$, where $g$ is given in \eqref{def g}. Then, an equivalent formulation of the problem in~\eqref{system_lin} with $p^0=v^0=0$ is given by the following coupled system: 
\small 
\begin{align}
%	\begin{aligned}
	m = 0:& \begin{cases}
		\begin{aligned}
			\text{(i)} & \quad p_0^N =0, \\
			\text{(ii)} 	& \quad \omega_0^2 v_0^N=  \frac{\ca}{2} \left( v_0^{N-1} \right)^2  + \sum_{j=0}^{N-1} \left(  \frac{\ca}{2} - \frac{\cb \omega^2}{2}j^2 \right)  \left| v_j^{N-1} \right|^2,\\
		\end{aligned} 
	\end{cases} \\[0.5em]
	m= 1:&  \begin{cases}
		\begin{aligned}
			\text{(i)} & \quad  -  p_1^N - \frac{ c^2 + \imath b \omega }{\omega^2} \Delta p_1^N+ c^2 \rho_0  n_0  v_1^{N}  = -h_1^N - \eta \sum_{k=3:2}^{2N-3}  \overline{p_{\frac{k-1}{2}}^{N-1}} p_{\frac{k+1}{2}}^{N-1}, \\[0.3em]
			\text{(ii)} & \quad  \frac{1}{\alpha_1}v_1^N + \mu p_1^{N} = \left( \ca - \cb \omega^2\right) v_0^{N-1}v_1^{N-1} 
			\\
			& \hspace{2.5cm} + \sum_{k=1:2}^{2N-3} \left( \ca - \cb \omega ^2 \frac{k^2 + 3}{4}\right) \overline{v_{\frac{k-1}{2}}^{N-1}} v_{\frac{k+1}{2}}^{N-1}, \\
		\end{aligned}
	\end{cases}  \\[0.5em]
	m = \left\lbrace 2, \ldots, N \right\rbrace:& \begin{cases}
		\begin{aligned}
			\text{(i)} & \quad - p_m^N- \frac{\left( c^2 + \imath m b \omega  \right) }{ \omega^2 m^2}\Delta p_m^N + c^2 \rho_0 n_0  v_m^{N}\\
			& \quad = -  h_m^N  - \frac{ \eta}{2} \left( \sum_{l=1}^{m-1}  p_l^{N-1} p_{m-l}^{N-1} -2 \sum_{k = m+2:2}^{2(N-1)-m}  \overline{p_{\frac{k-m}{2}}^{N-1}}  p_{\frac{k+m}{2}}^{N-1}  \right),  \\[0.3em]
			\text{(ii)}& \quad  \frac{1}{\alpha_m}  v_m^N + \mu p_m^{N} = \sum_{l=0}^m \left(\frac{\ca}{2} - \frac{\cb \omega^2}{2} (m-l)(2m-l) \right)  v_l^{N-1} v_{m-l}^{N-1} \\
			& \hspace{2.5cm} + \sum_{k=m:2}^{2(N-1)-m} \left( \ca \textcolor{black}{-} \frac{\cb \omega^2 (k^2 + 3 m^2)}{4} \right)  \overline{v_{\frac{k-m}{2}}^{N-1}} v_{\frac{k+m}{2}}^{N-1}
		\end{aligned}
	\end{cases}
	%	\end{aligned}
\end{align}
\normalsize 
for $N \geq 1$, where for $m \geq 1$, $\alpha_m = \left(- m^2\omega^2 + \imath m \delta \omega_0 \omega+\omega_0^2\right)^{-1}$.
Additionally, each of the individual functions $p_m^N$, where $m \in \left\lbrace 1, \ldots, N \right\rbrace$, should satisfy boundary conditions \eqref{eq: bc}.
\end{proposition}
\begin{proof}
The statement follows analogously to before by applying the Ansatz given in \eqref{finite_approx} and making use of the expressions provided in \cite[Lemma 3.1]{kaltenbacher2021periodic} for projections onto $\XN$ of products of two functions in $\XN$.
\end{proof}

Looking at the system in Proposition~\ref{prop: multilevel_linearized}, we see that we can express the volume harmonics as
\begin{equation} \label{volume harmonics}
 \vN_m = -\alpha_m \mu \pN_m + \alpha_m  f_v^{N-1} \text{ for } \ m \geq 1,
 \end{equation}
 where
\[
f^{N-1}_v = \sum_{l=0}^m \left(\frac{\ca}{2} - \frac{\cb \omega^2}{2} (m-l)(2m-l) \right)  v_l^{N-1} v_{m-l}^{N-1}+ \sum_{k=m:2}^{2(N-1)-m} \left( \ca \textcolor{black}{-} \frac{\cb \omega^2 (k^2 + 3 m^2)}{4} \right)  \overline{v_{\frac{k-m}{2}}^{N-1}} v_{\frac{k+m}{2}}^{N-1}.
\]
 By substituting these values into the equations for $\pN_m$, the coupled system in Proposition~\ref{prop: multilevel_linearized} can be rewritten as a system of \eqref{volume harmonics} and Helmholtz equations given by
	\begin{equation} \label{Helmholtz problems}
	\left\{	\begin{aligned}
		&-\left(1+\imath \omega\frac{m b}{c^2} \right)\Delta \pN_m - (k^2+\fraka) \pN_m +\imath \frakb\pN_m = -k^2 h_m^N - f_{p,v}^{N-1} \quad \text{ in }\, \Omega,\ \text{with } k= \frac{m \omega}{c},\\
		&(\imath \omega m \beta + \gamma ) \pN_m + \nabla \pN_m \cdot \vecn = 0 \quad \text{ on }\, \partial \Omega,
	\end{aligned} \right.
\end{equation}
where the source is
\begin{equation} \label{def fpvNone}
f_{p,v}^{N-1} =  \alpha_m \omega^2 m^2 \rho_0 n_0 f^{N-1}_v  + k^2 \frac{ \eta}{2} f^{N-1}_p
\end{equation}
with 
\[
f^{N-1}_p = \sum_{l=1}^{m-1}  p_l^{N-1} p_{m-l}^{N-1} -2 \sum_{k = m+2:2}^{2(N-1)-m}  \overline{p_{\frac{k-m}{2}}^{N-1}}  p_{\frac{k+m}{2}}^{N-1}.
\]
In \eqref{Helmholtz problems}, we have
\begin{equation} \label{deff fraka frakb}
\fraka =  \mu \rho_0 n_0 \frac{m^2 \omega^2 (\omega_0^2 - m^2 \omega^2)}{(\omega_0^2-m^2\omega^2)^2 + (m \delta \omega_0 \omega)^2}\ \text{  and  } \ \frakb = \mu \rho_0 n_0 \frac{m^3 \omega^3 \delta \omega_0 }{(\omega_0^2-m^2\omega^2)^2 + (m \delta \omega_0 \omega)^2}.
\end{equation}
Note that $\frakb>0$, while the sign of $\fraka$ depends on the relation between $\omega_0$ and $m \omega$. The form of the Helmholtz equation in \eqref{Helmholtz problems} reveals more clearly the influence of microbubbles on the wave propagation, through the dissipation signaled by the $i \frakb \pN_m$ term and modification of the wave number via the $-(k^2+\fraka) \pN_m$ term. \qed
\normalsize

\def\fNone{f^{N-1}} 
\def\fraka{\mathfrak{a}}
\def\frakb{\mathfrak{b}}
\section{Convergence of the linearized multiharmonic algorithm for real fields} \label{sec: convergence}
We next wish to establish convergence of the iterative scheme in \eqref{system_lin} as $N \rightarrow \infty$ in a suitable norm. We can mimic the proof of Theorem~\ref{thm: existence West-volume} to show by induction that for each $N \geq 1$, problem \eqref{system_lin} has a unique solution $(\pN, \vN) \in \XN \times \XN$. 
For fixed $N \geq 1$, the existence and uniqueness of $\pN$ follow by noting that, thanks to Proposition~\ref{prop: multilevel_linearized}, harmonics $\pN_m$ are obtained as solutions of the Helmholtz problems in \eqref{Helmholtz problems}.
	 Furthermore, an analogous testing procedure to the one in Theorem~\ref{thm: existence West-volume} applied on \eqref{system_lin} yields
	\[
	\|\pN\|_{\Xp} \leq \rp, \quad \|\vN\|_{\Xv} \leq \rv,
	\]
	provided $\rp$, $\rv$, and $\|\mu\|_{\Linf}$ are sufficiently small. We omit these details here and focus on establishing convergence. 
 The convergence will be shown in the norms of the spaces $\Xplowzero$ and $\Xvlowzero$, where
  \begin{equation}
  	\begin{aligned}
  	\Xplow^\ell =&\, H^{\ell+1}(0,T; H^1(\Omega)) \cap H^{\ell+2}(0,T; L^2(\Omega)) \cap H^{\ell+2}(0,T; L^2(\partial \Omega)),\\
  	 \Xvlow^\ell =&\, H^{\ell+2}(0,T; \Ltwo)
  	\end{aligned}
  \end{equation}
  for $\ell \geq 0$. To make writing more compact, we introduce the short-hand norm notation
  \begin{equation}
|\!|\!|(p,v)|\!|\!|_{\Xplowell \times \Xvlowell} \coloneqq	\|p\|_{\Xplowell}+\|v\|_{\Xvlowell}
  \end{equation}
  for $(p,v) \in \Xplowell\times \Xvlowell$, and we also adopt the operator notation
  \begin{equation}
  	\begin{aligned}
  		<\calL_1 \pN, \phiN>  \coloneqq \begin{multlined}[t]
  			\intTO \left(\pttN \phiN + c^2 \nabla \pN \cdot \nabla \phiN+ b \nabla \ptN \cdot \nabla \phiN \right)\dxs \\
  			+ \intT \int_{\partial \Omega} (c^2(\beta\ptN+\gamma \pN)+b(\beta \pttN+\gamma \ptN))\phiN \dGs
  		\end{multlined}
  	\end{aligned}
  \end{equation}
  and
  \begin{equation}
  	<\calL_2 \vN, \phiN>  \coloneqq \intTO \left(\vttN + \delta \omega_0 \vtN + \omega_0^2 \vN \right) \phiN \dxs, \quad \phiN \in \XN.
  \end{equation}
  Then the algorithm can be restated as follows. Given the previous iterate  $(\pNone,  \vNone) \in \XN \times \XN$, we compute $(\pN, \vN)$ for $N \geq 1$ as the solution of
  	\begin{subequations}\label{iterative method}
  		\begin{equation}	\label{eq pN iterative}
  			\begin{aligned}
  				<\calL_1 \pN, \phiN>	=&\, \intTO \left({\eta} ((\pNone)^2)_{tt} + c^2 \rho_0 n_0 \vttN +h\right) \phiN \dxs,
  				\end{aligned}
  		\end{equation}	
  		and
  			\begin{equation} \label{eq vN iterative}
  					\begin{aligned}
  						<\calL_2 \vN, \phiN> 	=&\,\begin{multlined}[t] \intTO \Bigl\{-\mu\pN + \ca (\vNone)^2  \\  \hspace*{2cm}+ \cb \left( 2 v^{N-1} \vttNone + (\vtNone)^2 \right) \Bigr\} \phiN \dxs,
  								\end{multlined}
  						\end{aligned} 
  				\end{equation}
  		\end{subequations}	
  		for all $\phiN \in \XN$. \\
  \indent In the convergence proof, we will involve the truncated Fourier series of a continuous-in-time function $u$ given by 
  \begin{equation}
  	\begin{aligned}
  		\tilde{u}^N(x,t)= \sum_{m=0}^N \left [u^{c}_m (x) \cos(m \omega t)+ u_m^{s}(x)\sin(m \omega t)\right ] \in \XN,  \quad N \geq 0,
  	\end{aligned}
  \end{equation}
  with the coefficients 
  \begin{equation}
  	\begin{aligned}
  		u_m^{c}(x) = \frac{2}{T} \int_0^T u(x,t) \cos(m \omega t) \dt, \quad u_m^{s}(x) = \frac{2}{T} \int_0^T u(x,t) \sin(m \omega t) \dt.
  	\end{aligned}
  \end{equation}
By integrating the time-periodic Westervelt equation and acoustic boundary conditions in \eqref{ibvp:westervelt volume periodic} over $(0,T)$, we can conclude that $\int_0^T p(x,t) \dt =0$ in $\Omega$.  Thus  $\tpzero=0$, which is reflected by having $\pN_0=0$ in the multiharmonic algorithms. \\
\indent It can be shown analogously to~\cite[Lemma 12]{bachinger2005numerical} that the following error bound holds for $\ell \geq 1$:
\begin{equation} \label{conv Fourier}
	\begin{aligned}
|\!|\!|(p-\tpN, v-\tvN)|\!|\!|_{(\Xplowzero \cap H^2(\Hone)) \times \Xvlowzero}	
	 \leq&\, C N^{-\ell} (\|p\|_{Y_p^{\ell} \cap H^{\ell+2}(\Hone)}+\|v\|_{Y_v^{\ell}}),
	\end{aligned}
\end{equation}
which we can exploit to characterize the error of the scheme in the next statement. We include the proof of \eqref{conv Fourier} in Appendix~\ref{Appendix: Cutoff est} for completeness.

 \begin{theorem}
 	\label{thm: convergence}
 Let the assumptions of Theorem~\ref{thm: existence West-volume} hold  with the acoustic source term $h=g_{tt}$, where $g$ is given in \eqref{def g}. Assume additionally that
 \begin{equation}
 	\begin{aligned}
 		(p,v) \in \left(\Xp \cap \Xplow^\ell \cap H^{\ell+2}(0,T; \Hone) \right) \times \left(\Xv \cap \Xvlow^\ell\right), \quad \ell \geq 1.
 	\end{aligned}
 \end{equation}
 Let $p^0=v^0=0$ and assume that $(\pN, \vN)$ are computed using \eqref{iterative method} for $N \geq 1$.  Then, provided $\rp$, $\rv$, and $\|\mu\|_{\Linf}$ are sufficiently small, there exists $q=q(\rp, \rv)<1$, such that
 \begin{equation} \label{error est pN vN}
 	\begin{aligned}
 		&\pvnorm{(p-\pN, v-\vN)} \\
 	\leq&\, \begin{multlined}[t] 	q \pvnorm{(p-\pNone, v-\vNone)}
 		+  C N^{-\ell} (\|p\|_{\Xplowell \cap H^{\ell+2}(\Hone)}+\|v\|_{\Xvlowell}),
 		\end{multlined}
 	\end{aligned}
 \end{equation}
 where $C>0$ does not depend on $N$.
 \end{theorem}
	\begin{proof}
		 The proof follows by splitting the error as follows:
		\begin{equation}
			\begin{aligned}
				p -\pN =&\, (p- \tpN)- (\pN- \tpN) \eqqcolon \errprojNp-\errpN, \\
				 	v -\vN =&\, (v-\tvN) -(\vN-\tvN) \eqqcolon \errprojNv-\errvN,
			\end{aligned}
		\end{equation}
		and then representing $(\errpN, \errvN)$
		as the solution of a suitable semi-discrete wave-volume system. The discrete error $(\errpN, \errvN)$ can be seen as the solution of
\begin{subequations}
 \begin{equation} \label{eq errpN}
	\begin{aligned}
		<\calL_1 \errpN, \phiN> 
		=&\, \begin{multlined}[t]-\intTO \left({\eta} (p^2-(\pNone)^2)_{tt} + c^2 \rho_0 n_0 (\vtt-\vttN)\right) \phiN \dxs \\+ <\calL_1 \errprojNp, \phiN>
			\end{multlined}
	\end{aligned}
\end{equation}	
		 and
	 \begin{equation} \label{eq errvN}
	\begin{aligned}
		& <\calL_2 \errvN, \phiN> \\
		=&\,\begin{multlined}[t] -\intTO \Bigl(-\mu (p-\pN)+\ca (v^2-(\vNone)^2) \\ \hspace*{2cm} + \cb ( 2(v \vtt - v^{N-1} \vttNone) +(\vt^2- (\vtNone)^2))\Bigr) \phiN \dxs\\
			 + <\calL_2 \errprojNv, \phiN>
			 \end{multlined}
	\end{aligned}
\end{equation}
\end{subequations}
	for all $\phiN \in \XN$. Note that in the testing procedure for this problem we can exploit the fact that $\int_0^T \ddt (\cdot) \dt=0$ for time-periodic functions. By testing \eqref{eq errpN} with $\errpN$, $\errptN$, and $\errpttN$, employing Young's and H\"older's inequalities, and combining the resulting estimates, we can derive the following bound: 
		 \begin{equation} \label{error est}
		 	\begin{aligned}
		 		\|\errpN\|^2_{\Xplowzero}
		 		\lesssim&\, \begin{multlined}[t]
		 			\|{-\eta} (p^2-(\pNone)^2)_{tt}- c^2 \rho_0 n_0 (\vtt-\vttN)\|^2_{L^2(\Ltwo)} 	\\	 			+ \|\errprojNp\|^2_{\Xplowzero \cap H^{2}(\Hone)},
		 		\end{multlined} 
		 	\end{aligned}
		 \end{equation}		 
where we have also exploited the equivalence of the norms $\|w\|_{H^1(\Omega)}$ and $\|\nabla w\|_{L^2(\Omega)}+\|w\|_{L^2(\partial \Omega)}$ (see~\cite[Theorem 1.9]{necas2011direct}). The derivation of \eqref{error est} is provided in Appendix~\ref{Appendix: Proof error energy est}.\\
\indent Then by using the rewriting
\begin{equation}
	\begin{aligned}
		((\pNone)^2)_{tt}-(p^2)_{tt} 
		=\, 2(\pNone_t-\pt)(\pNone_t+\pt)+2(\pNone-p)\pNone_{tt}+2 p(\pNone_{tt}-\ptt),
	\end{aligned}
\end{equation}
and the fact that $\|p\|_{\Xp}$, $\|\pN\| \leq \rp$ for all $N \geq 1$, we obtain
	\begin{equation}
	\begin{aligned}
		\|\errpN\|_{\Xplowzero} \lesssim&\,\begin{multlined}[t] \rp \|\pNone-p \|_{\Xplowzero}+ {\| n_0 \|_{\Linf}} \|\vttN-\vtt\|_{\LtwoLtwo}
			+ \|\errprojNp\|_{\Xplowzero \cap H^{2}(\Hone)}.
			\end{multlined}
	\end{aligned}
\end{equation}
 From here via $p-\pN=\errprojNp-\errpN$ and the triangle inequality, we arrive at
	\begin{equation} \label{interim est1}
	\begin{aligned}
		\|p-\pN\|_{\Xplowzero} \lesssim&\,\begin{multlined}[t]  \rp \|\pNone-p \|_{\Xplowzero}+ {\| n_0 \|_{\Linf}} \|\vttN-\vtt\|_{\LtwoLtwo}\\
			+ \|\errprojNp\|_{\Xplowzero \cap H^{2}(\Hone)}.
		\end{multlined}
	\end{aligned}
\end{equation}
Similarly, testing \eqref{eq errvN} with $\errvN$, $\errvtN$, and $\errvttN$ yields, after standard manipulations,
\begin{equation}
	\begin{aligned}
		\|\errvN\|^2_{\Xvlowzero} 
		\lesssim&\, \begin{multlined}[t]
		\|p-\pN \|^2_{L^2(\Ltwo)}+\|v^2-(\vNone)^2\|^2_{L^2(\Ltwo)}\\
		\| - \cb ( 2(v \vtt - v^{N-1} \vttNone) +(\vt^2- (\vtNone)^2))\|^2_{L^2(\Ltwo)}
		+ \|\errprojNv\|^2_{\Xvlowzero}.
		\end{multlined}
	\end{aligned}
\end{equation}
By using the fact that $\|v\|_{\Xv}$, $\|\vN\|_{\Xv} \leq \rv$ for all $N \geq 1$, and the triangle inequality, from \eqref{eq errvN} we then have
\begin{equation} \label{interim est 2}
	\begin{aligned}
		\|v-\vN\|_{\Xvlowzero} 
		\lesssim&\, \begin{multlined}[t]
\|\pN-p\|_{L^2(\Ltwo)}+\rv \|\vNone-v\|_{\Xvlowzero}
+\|\errprojNv\|_{\Xvlowzero}.
		\end{multlined}
	\end{aligned}
\end{equation}
Adding $\lambda \cdot$\eqref{interim est 2} and \eqref{interim est1}  with $\lambda>0$ small enough, so that the term $\lambda \|\pN-p\|_{L^2(\Ltwo)}$ can be absorbed by the left-hand side, and then possibly reducing $\|n_0\|_{\Linf}$ so that the term $\|n_0\|_{\Linf} \|\vttN-\vtt\|_{\LtwoLtwo}$ can be absorbed, yields
\begin{equation} \label{interim}
	\begin{aligned}
		&\|p-\pN\|_{\Xplowzero} + \lambda\|v-\vN\|_{\Xvlowzero}  \\
		 \lesssim&\,\begin{multlined}[t]
		  \rp \|\pNone-p \|_{\Xplowzero}
		+ \lambda	\rv \|\vNone-v\|_{\Xvlowzero}
		 	+ \|\errprojNp\|_{\Xplowzero \cap H^{2}(\Hone)}+\lambda \|\errprojNv\|_{\Xvlowzero}.
		\end{multlined} 
	\end{aligned}
\end{equation}
By possibly further reducing $r_p$ and $\rv$, we obtain $q(\rp, \rv) <1$, such that %. Thus from \eqref{interim} we obtain
\begin{equation} \label{interim est6}
	\begin{aligned}
	\pvnorm{(p-\pN, v-\vN)} 
		\leq&\, \begin{multlined}[t] 	q(\rp, \rv) \pvnorm{(p-\pNone, v-\vNone)}\\
			+  (1+\|n_0\|_{\Linf})C_0\pvnormcap{(\errprojNp, \errprojNv)}
		\end{multlined}
	\end{aligned}
\end{equation}
for some $C_0>0$. The statement then follows by employing  \eqref{conv Fourier}.\qed
\end{proof}
Convergence of the scheme then  follows in a straightforward manner from \eqref{error est}.
\begin{corollary}[Convergence of the linearized scheme]
 Let the assumptions of Theorem~\ref{thm: convergence} hold. Then the solution $(\pN, \vN) \in \XN \times \XN$ of the iteration scheme \eqref{iterative method} converges with respect to the norm $\pvnorm{(\cdot, \cdot)}$ to the solution $(p, v) \in \ballr$ of  \eqref{ibvp:westervelt volume periodic} as $N \rightarrow \infty$. 
 \end{corollary}
 \begin{proof}
 From \eqref{error est pN vN}, by iteration we obtain 
 	\begin{equation} \label{interim est7}
 		\begin{aligned}
 			\pvnorm{(p-\pN, v-\vN)} 
 			\leq&\, \begin{multlined}[t] 	 q(\rp, \rv)^N \pvnorm{(p-p^0, v-v^0)}
 				\\\hspace*{-0.5cm}+  C\sum_{i=1}^N 	q(\rp, \rv)^{N-i} i^{-\ell} (\|p\|_{\Xplowell \cap H^{\ell+2}(\Hone)}+\|v\|_{\Xvlowell}),
 			\end{multlined}
 		\end{aligned}
 	\end{equation}
 	for some $C>0$, independent of $N$. The statement then follows analogously to~\cite[Theorem 3.1]{Rainer2024nonlinear}, where the multiharmonic discretization of the de-coupled Westervelt equation is studied. We thus omit the details here. \qed
 \end{proof}
 \subsection{Setting  the zeroth microbubble harmonic to zero} 
  Unlike with the Westervelt equation, from the equation for $v$, we cannot in general conclude by integrating from $0$ to $T$ that $\int_0^T v(t) \dt =0$. Nevertheless, in numerical methods for solving multiharmonic problems it is often assumed that the zeroth harmonic does not contribute significantly to the dynamics; this assumption is also made in the formal two-harmonic generation study for the linearized wave-ODE model with $b=\eta=0$ in~\cite[Ch.\ 5.3.2]{hamilton1998nonlinear}. \\
 \indent With the assumption $p_0^N = v_0^N = 0$ for all $N \geq 0$, the system derived in Proposition \ref{prop: multilevel_linearized} further simplifies to
 \small
 \begin{align} \label{eq: multi_sim1}
 		m= 1:&  \begin{cases}
 			\begin{aligned}
 				\text{(i)} & \quad  -  p_1^{N} - \frac{ c^2 + \imath b \omega }{\omega^2} \Delta p_1^N+ c^2 \rho_0  n_0  v_1^{N}  = -h_1^N - \sum_{k=3:2}^{2N-3}  \eta  \overline{p_{\frac{k-1}{2}}^{N-1}} p_{\frac{k+1}{2}}^{N-1} \\[0.3em]
 				\text{(ii)} & \quad  v_1^N = \alpha_1 \left( - \mu p_1^{N} +  \sum_{k=3:2}^{2N-3} \left( \ca - \cb \omega ^2 \frac{k^2 + 3}{4}\right) \overline{v_{\frac{k-1}{2}}^{N-1}} v_{\frac{k+1}{2}}^{N-1}\right) \\
 			\end{aligned}
 		\end{cases}  \\[0.5em]
 		m = \left\lbrace 2, \ldots, N \right\rbrace:& \begin{cases}
 			\begin{aligned}
 				\text{(i)} & \quad - p_m^N- \frac{\left( c^2 + \imath m b \omega  \right) }{ \omega^2 m^2}\Delta p_m^N + c^2 \rho_0 n_0  v_m^{N}\\
 				& \quad = -  h_m^N  - \frac{ \eta}{2} \left( \sum_{l=1}^{m-1}  p_l^{N-1} p_{m-l}^{N-1} -2 \sum_{k = m+2:2}^{2(N-1)-m}  \overline{p_{\frac{k-m}{2}}^{N-1}}  p_{\frac{k+m}{2}}^{N-1}  \right) \\[0.3em]
 				\text{(ii)}& \quad  v_m^N  = \alpha_m \left( -\mu p_m^{N} + \sum_{l=1}^{m-1} \left(\frac{\ca}{2} - \frac{\cb \omega^2}{2} (m-l)(2m-l) \right)  v_l^{N-1} v_{m-l}^{N-1} \right. \\
 				& \quad \left. + \sum_{k=m+2:2}^{2(N-1)-m} \left( \ca \textcolor{black}{-} \frac{\cb \omega^2 (k^2 + 3 m^2)}{4} \right)  \overline{v_{\frac{k-m}{2}}^{N-1}} v_{\frac{k+m}{2}}^{N-1} \right) \\
 			\end{aligned}
 		\end{cases}
 		%	\end{aligned}
 \end{align}
 \normalsize 
 where $\alpha_m = \left(- m^2 \omega^2 + \imath m \delta \omega_0 \omega+ \omega_0^2 \right)^{-1}$ for $m \in \left\lbrace 1, \ldots N\right\rbrace$. 
Now $v_m^N$ can be eliminated explicitly by substituting the second equation into the first, leading to a system of inhomogeneous Helmholtz equations. 

\section{Time discretization via a multiharmonic Ansatz for complex fields} \label{sec: multiharmonic complex}

In this section, we formally investigate multiharmonic algorithms based on neglecting the fact that the excitation and solution fields should be real-valued. We look for approximate solutions in 
\[
	\tilde{X}_N := \left\lbrace \sum_{m=0}^N \exp( \imath m \omega t) u_m^N(x) \, : \, u_m^N \in H^2(\Omega; \C) \right\rbrace.
\]
 This approach, adopted also in the investigation of the (de-coupled) Westervelt equation in~\cite{kaltenbacher2021periodic}, will allow us to arrive at a simple triangular approximation scheme in the frequency domain.  To this end, we simply set 
\begin{equation} \label{def g simple}
	g^M(x,t) = \displaystyle \sum_{m=0}^{M}\exp(\imath  \omega m t) h_m^M(x), \quad \omega = \frac{2 \pi}{T},
\end{equation}
and make the following Ansatz:
\begin{equation} \label{eq: ansatz_linear}
	\pN(x,t)= \sum_{m=0}^{N} \exp(\imath m \omega t) p_m^N(x), \qquad v^N(x,t)= \sum_{m=0}^{N} \exp(\imath m \omega t) v_m^N(x),
\end{equation}
where the coefficients $\pN_m$, $\vN_m \in H^2(\Omega; \C)$ are determined by solving \eqref{semi-discrete} with $\projN(\cdot)$ replaced by $\projNt(\cdot)$. 

\subsection{A multiharmonic cut-off algorithm} We next set up a simplified algorithm (compared to the one in Proposition \ref{prop: multi}) in the sense that the right-hand side in the resulting system will contain only one summation term, such that there is reduced coupling across harmonics, making it computational more efficient. 

In this complex setting, for $v_0^N$, we obtain the quadratic equation
\begin{equation}
	\omega_0^2 v_0^N - \ca (v_0^N)^2 = 0
\end{equation}
and choose the zero solution, such that we have $v_0^N=0$ as before in \eqref{eq: multi_sim1}. This choice is consistent with the formulation used throughout the numerical simulations and avoids introducing spurious constant components.  

\begin{proposition}\label{prop: simplified setting}
	Let the assumptions of Theorem~\ref{thm: existence West-volume} hold  with the acoustic source term $h=g_{tt}$, where $g$ is given in \eqref{def g simple}. 
Under the Ansatz in \eqref{eq: ansatz_linear} with $\pN_0=\vN_0=0$, the problem in \eqref{semi-discrete} with $\projN(\cdot)$ replaced by $\projNt(\cdot)$ can be rewritten as follows:
	\begin{align} \label{multiharm_complex}
			m= 1:&  \begin{cases}
				\begin{aligned}
					\text{(i)} & \quad   - p_1^N- \frac{c^2 + \imath \omega b }{\omega^2} \Delta p^N_1 + c^2 \rho_0 n_0   v_1^N  = -h_1^N  \\[0.3em]
					\text{(ii)} & \quad \frac{1}{\alpha_1} v_1^N + \mu p_1^N  =0\\
				\end{aligned}
			\end{cases}  \\[0.5em]
			m = \left\lbrace 2, \ldots, N \right\rbrace:& \begin{cases}
				\begin{aligned}
					\text{(i)} & \quad -  p_m^N - \frac{(c^2 + \imath m \omega b)}{m^2 \omega^2} \Delta p_m^N + c^2 \rho_0 n_0 v_m^N  =- h_m^N- \eta \sum_{l=1}^{m-1}      p_l^N p_{m-l}^N     \\[0.3em]
					\text{(ii)} & \quad  \frac{1}{\alpha_m} v_m^N + \mu p_m^N= \sum_{l=1}^{m-1} \left( \ca - \cb \omega^2 (m-l)(2m-l) \right) v_{l}^Nv^{N}_{m-l},
				\end{aligned}
			\end{cases}
			%	\end{aligned}
	\end{align}
where $\alpha_m = \left(-m^2\omega^2 + \imath m \delta \omega_0 \omega + \omega_0^2\right)^{-1}$ for $m \in \left\lbrace 1, \ldots N\right\rbrace$. 
\normalsize 
	Additionally, each of the individual functions $p_m^N$, where $m \in \left\lbrace 0, \ldots, N \right\rbrace$, should satisfy boundary conditions \eqref{eq: bc}.
\end{proposition}

\begin{proof}
	Plugging the approximation given in \eqref{eq: ansatz_linear} into \eqref{ibvp:westervelt volume periodic} yields
	\begin{align}
			\sum_{m=0}^{N} & \left[ - m^2 \omega^2 p_m^N- (c^2 + \imath m \omega b) \Delta p_m^N + c^2 \rho_0 n_0 m^2 \omega^2 v_m^N +   m^2\omega^2 h_m^N\right. \\
			& \left.  \qquad \qquad+  \sum_{l=0}^m \eta m^2 \omega^2 p_l^N p_{m-l}^N  \right] \exp(\imath m \omega t)= 0.
	\end{align}
	\normalsize 
	Concerning the nonlinear term on the right hand-side of the ODE, we obtain
	\begin{align}
		\text{Proj}_{\tilde{X}_N}(2 \vN \vN_{tt} + (\vN_t)^2) &= - \sum_{m=0}^N \sum_{l=0}^m \left( 2(m-l)^2 + l (m-l) \right) \omega^2 v_l^N v_{m-l}^N \exp(\imath m \omega t) \\
		& = - \sum_{m=0}^N \sum_{l=0}^m (m-l)(2m-l)  \omega^2 v_l^N v_{m-l}^N \exp(\imath m \omega t).
	\end{align}
Therefore, the approximation for the ODE is given by
	\begin{align}
			\sum_{m=0}^{N} & \left[ \vphantom{\sum_{l=0}^m} \left( -m^2 \omega^2 + \delta \omega_0 \imath m \omega + \omega_0^2 \right)v_m^N + \mu p_m^N \right. \\
			& \left.+ \sum_{l=0}^m \left( - \ca + \cb \omega^2 (m-l)(2m-l)\right) v_l^N v_{m-l}^N \right] \exp(\imath m \omega t) = 0.
	\end{align}
	\normalsize 
	Using linear independence of the functions $t \mapsto \exp(\imath m \omega t)$, we immediately arrive at the iterative coupled system.\qed
\end{proof}

\subsection{A linearized multiharmonic cut-off algorithm} 
We next also want to look at the linearized set-up given in \eqref{system_lin} (with $\projN(\cdot)$ replaced by $\projNt(\cdot)$) in the setting of complex solution fields. In this setting, for $u^N \in \tilde{X}_N$, we formally define $u_m^N = 0$ for all $m>N$.

\begin{proposition}
	Let the assumptions of Theorem~\ref{thm: existence West-volume} hold  with the acoustic source term $h=g_{tt}$, where $g$ is given in \eqref{def g simple}. 	 
	Under the Ansatz in \eqref{eq: ansatz_linear} with $\pN_0=\vN_0=0$, the problem in \eqref{system_lin}, with $\projN(\cdot)$ replaced by $\projNt(\cdot)$, can be rewritten as follows:
	\begin{align} \label{system_v_0_0}
		\begin{aligned}
				m= 1:&  \begin{cases}
						\begin{aligned}
								\text{(i)} & \quad   - p_1^N- \frac{c^2 + \imath \omega b }{\omega^2} \Delta p^N_1 + c^2 \rho_0 n_0   v_1^{N}  = -h_1^N,  \\[0.3em]
								\text{(ii)} & \quad \frac{1}{\alpha_1} v_1^N + \mu p_1^{N} =0, \\
							\end{aligned}
					\end{cases}  \\[0.5em]
				m = \left\lbrace 2, \ldots, N \right\rbrace:& \begin{cases}
						\begin{aligned}
								\text{(i)} & \quad -  p_m^N - \frac{(c^2 + \imath m \omega b)}{m^2 \omega^2} \Delta p_m^N + c^2 \rho_0 n_0 v_m^{N}  =- h_m^N-\eta \sum_{l=1}^{m-1}      p_l^{N-1} p_{m-l}^{N-1},     \\[0.3em]
								\text{(ii)} & \quad  \frac{1}{\alpha_m} v_m^N + \mu p_m^{N}  = \sum_{l=1}^{m-1} \left(\ca  - \cb \omega^2 (m-l)(2m-l) \right) v_{l}^{N-1}v^{N-1}_{m-l},
							\end{aligned}
					\end{cases}
			\end{aligned}
	\end{align} 
where $\alpha_m = \left(-m^2\omega^2 + \imath m \delta \omega_0 \omega + \omega_0^2\right)^{-1}$ for $m \in \left\lbrace 1, \ldots N\right\rbrace$. Additionally, each of the individual functions $p_m^N$ for $m \in \left\lbrace 0, \ldots, N \right\rbrace$ is supposed to satisfy boundary conditions \eqref{eq: bc}.
		%	\end{equation}
\end{proposition}

\begin{proof}
	Similarly to the previous case, substituting the approximation provided in \eqref{eq: ansatz_linear} into \eqref{ibvp:westervelt volume periodic} immediately yields the iterative coupled system.\qed
\end{proof}
We see that when working with complex fields from the beginning, we end up with a lower triangular system of inhomogeneous Helmholtz equations and algebraic equations that can be solved by substitution. 
\begin{remark}[On the assumption of complex fields]
		In this section, the approximations $(p^N,v^N)$ are complex-valued, so convergence could only be studied in the complex counterpart of $X_N\times X_N$. This does not by itself imply that $\Re(p^N,v^N)\to(p,v)$ as $N\to\infty$; such a conclusion would require enforcing the reality constraint (symmetric $\pm m$ spectrum with $\hat u_{-m}=\overline{\hat u_m}$). Following \cite{bachinger2005numerical}, one may formulate the multiharmonic Ansatz with complex Fourier coefficients (omitting the explicit $\Re$ in the time reconstruction), but the resulting nonlinear maps are not complex-differentiable (holomorphic). Hence a Newton scheme in the sense of complex analysis is not applicable; in practice one uses the real formulation (or a split into real/imaginary parts) despite the algebraic brevity of the complex notation.\\
		\indent	For clarity, eq.~\eqref{eq: multi_sim1} (real basis) contains both sum- and difference-frequency couplings generated by the quadratic terms, visible in the two convolution sums $ \sum_{l=1}^{m-1} p_l\,p_{m-l}$ and $ \sum_{k=m+2:2}^{2(N-1)-m} p_{(k-m)/2}\,p_{(k+m)/2}$.
		In contrast, the complex projection \eqref{system_v_0_0} is posed on the half-spectrum $\{e^{\mathrm i m\omega_0 t}\}_{m=1}^N$ without imposing $\hat{u}_{-m}=\overline{\hat{u}_m}$; only the sum-frequency interactions remain and the block system becomes lower-triangular (solvable by substitution). If one augments the complex space to include $\pm m$ and enforces the reality constraint, the complex and real formulations are algebraically equivalent. In our parameter regime, the neglected difference-frequency contributions are numerically negligible, which explains the agreement observed in Fig.~\ref{fig: simplified_vs_real} in Section \ref{sec:numerics}. %\vanja{Should we change the reference to a different figure}
\end{remark}
\subsection{A simple two-harmonic approximation scheme}  In certain imaging applications, the fundamental frequency and the second harmonic are considered the most significant; see, for example,~\cite{neer2011}. We thus compare the multiharmonic iterative scheme with a simplified two-harmonic scheme to evaluate the extent of information loss.

For two harmonics, that is, for $N=2$ with $p_0^{(2)}=0$ and $v_0^{(2)}=0$, the approximations for the pressure $p$ and the volume $v$ are given by
\begin{equation}\label{Ansatz: two harmonic}
	\begin{aligned} 
	p(x,t) \approx	p^{(2)}(x,t) &= p_1^{(2)}(x) \exp(\imath \omega t) + p_2^{(2)}(x) \exp(2 \imath \omega t), \\ v(x,t) \approx v^{(2)}(x,t) &= v_1^{(2)}(x) \exp(\imath \omega t) + v_2^{(2)}(x) \exp(2 \imath \omega t). 
	\end{aligned}
\end{equation}

The corresponding scheme for computing $p_{1,2}^{(2)}$ and $v_{1,2}^{(2)}$ can be found by setting $N=2$ in Proposition \ref{prop: simplified setting}. This results in the system
\begin{align} \label{iterative_two_harmonics}
	\begin{aligned}
		m= 1:&  \begin{cases}
			\begin{aligned}
				\text{(i)} & \quad   p_1^{(2)}+ \frac{c^2 + \imath \omega b }{\omega^2} \Delta p^{(2)}_1 - c^2 \rho_0 n_0   v_1^{(2)}  = h_1^{(2)},  \\[0.3em]
				\text{(ii)} & \quad \frac{1}{\alpha_1} v_1^{(2)} + \mu p_1^{(2)}  =0,\\
			\end{aligned}
		\end{cases}  \\[0.5em]
		m = 2:& \begin{cases}
			\begin{aligned}
				\text{(i)} & \quad p_2^{(2)}+ \frac{c^2 + 2 \imath \omega b}{4 \omega^2} \Delta p_2^{(2)} - c^2 \rho_0 n_0 v_2^{(2)}  = h_2^{(2)} +\frac{ \beta_a }{ \rho_0 c^2} (p_1^{(2)})^2,   \\[0.3em]
				\text{(ii)} & \quad  \frac{1}{\alpha_2} v_2^{(2)} + \mu p_2^{(2)} = \left( \ca - 3 \cb \omega^2 \right) (v_1^{(2)})^2,
			\end{aligned}
		\end{cases}
	\end{aligned}
\end{align}
where, as before,
\begin{equation}
	\alpha_m = \frac{1}{- m^2 \omega^2 + \imath m \delta \omega_0  \omega + \omega_0^2} = \frac{ \omega_0^2- m^2 \omega^2}{ (\omega_0^2- m^2 \omega^2 )^2+  (m \delta \omega_0  \omega)^2} -\imath  \frac{  m \delta \omega_0  \omega}{ (\omega_0^2- m^2 \omega^2 )^2+  (m \delta \omega_0  \omega)^2}
\end{equation}
for $m \in \{1,2\}$.  
Above, we have rewritten $\eta$ as $\eta = \dfrac{ \beta_a }{ \rho_0 c^2}$ to make the influence of the nonlinearity parameter in the medium $\beta_a$ on the system explicit. In this two-harmonic setting, we can express $v_1^{(2)}$ in terms of $p_1^{(2)}$ and $v_2^{(2)}$ in terms of $p_2^{(2)}$:
\begin{equation}\label{two harmonics ODE}
	\begin{aligned}
		v_1^{(2)} &= - \alpha_1 \mu \, p_1^{(2)}, \\
		v_2^{(2)} &= \alpha_2 \left[ - \mu p_2^{(2)} + (\ca-3\cb\omega^2) {\alpha}_1^2 \mu^2 \left( p_1^{(2)}\right)^2 \right].
	\end{aligned}
\end{equation}
We can then use these expressions to eliminate $v_1^{(2)}$ and $v_2^{(2)}$ for the first and third equations in \eqref{iterative_two_harmonics}. Multiplying the resulting equations by  $m^2 \omega^2/c^2$ yields the following system for the two pressure coefficients:
\begin{align} \label{two-harmonic: nonlinear}
	\begin{aligned}
		m = 1: \qquad &\frac{ \omega^2}{\tilde{c}_1^2} p_1^{(2)} +\left( 1 + \imath \frac{ b \omega}{c^2}  \right)  \Delta p_1^{(2)} = \frac{\omega^2}{c^2} h_1^{(2)},\\
		m = 2: \qquad & \frac{4 \omega^2}{\tilde{c}_2^2} p_2^{(2)} + \left( 1 +\imath \frac{2 b \omega}{c^2} \right) \Delta p_2^{(2)} = \frac{4 \omega^2}{c^2} h_2^{(2)} + \frac{4 \omega^2 }{\rho_0 c^4 } \left( \beta_a +  \tilde{\beta}_a\right)  \left( p_1^{(2)}\right) ^2,  
	\end{aligned}
\end{align}
where 
\begin{equation} \label{def cm}
	\frac{1}{\tilde{c}_{m}^2} =  \frac{1}{c^2} + \rho_0 n_0  \mu \alpha_{m}, \qquad 	\tilde{\beta}_a= c^4 \rho_0^2  n_0(\ca-3\cb \omega^2)\mu^2 \alpha_1^2 \alpha_2.
\end{equation}
The above system clearly displays the influence of microbubbles on wave propagation, particularly through their impact on the effective wave speed and the nonlinearity parameter.
\begin{itemize}[leftmargin=*]
	\item From equations \eqref{two-harmonic: nonlinear}, we see that the effective wave number in the presence of microbubbles in a two-harmonic setting is 
		\begin{equation}
				\tilde{k}_m= \frac{m \omega}{\tilde{c}_m}, \quad \text{with } \tilde{c}_m \ \text{ given in } \eqref{def cm}.
			\end{equation}
	Note that $\tilde{k}_m$ is in general complex since $\alpha_m$ is complex, and thus the attenuation in the wave propagation is due both to the acoustic dissipation via the $\imath \frac{m b \omega}{c^2} \Delta p_m^{(2)}$ terms as well as the presence of microbubbles via the $\imath \Im(\tilde{k}^2_m) p_m^{(2)}$ terms. \\[1mm]
	\item  In the second equation in \eqref{two-harmonic: nonlinear}, we see the influence of microbubbles on the nonlinearity parameter as their presence leads to the additional $\tilde{\beta}_a$ term on the right-hand side, thus effectively  enlarging the nonlinearity parameter $\beta_a$. By setting $\tilde{\beta} = 0$ and replacing $\tilde{c}^2_{m}$ with $c^2$, the system reduces to the two-harmonic expansion of Westervelt's equation for a strongly damped medium without bubbles. 
\end{itemize}
\def\fomega{f_\omega}
\section{Numerical experiments}\label{sec:numerics}
In this section, we investigate numerically the algorithms set up in Sections \ref{sec: multiharmonic} and \ref{sec: multiharmonic complex}. To discretize the resulting Helmholtz problems, we use a conforming finite element method. We note that algorithmic frameworks of the multiharmonic-FEM type used in this section are known as the Harmonic Balance Finite Element Method (HBFEM) in the context of electromagnetism; see, e.g.,~\cite{HBFEM2016, bachinger2005numerical, bachinger2006efficient}. \\
\indent The discretization of the classical Helmholtz equation $-\Delta u-k^2 u=f$ with $k>0$, and the issues related to the so-called pollution effect and sign-indefiniteness have been extensively studied in the literature; see, e.g., \cite{babuska2000, ihlenburg1995finite, melenk2010, moiola2014helmholtz}, and the references contained therein. Recall that the Helmholtz problems appearing in  linearized problems in \eqref{eq: multi_sim1} and \eqref{system_v_0_0} have the general form 
\begin{equation} \label{Helmholtz problems numerics}
\left\{	\begin{aligned}
	&-\left(1+\imath \omega\frac{m b}{c^2} \right)\Delta u - (k^2+\fraka) u + \imath \mathfrak{b} u= \fomega \quad \text{ in }\, \Omega,\quad \text{with } k= \frac{m \omega}{c},\\
	&(\imath \omega m \beta + \gamma ) u + \nabla u \cdot \vecn = 0 \quad \text{ on }\, \partial \Omega,
	\end{aligned} \right.
	\end{equation}
for a known right-hand side $\fomega \in L^2(\Omega; \C)$ (that depends on $\omega$), and with
\begin{equation}
	\fraka = \mu \rho_0 n_0 \frac{m^2 \omega^2 (\omega_0^2 - m^2 \omega^2)}{(\omega_0^2-m^2\omega^2)^2 + (m \delta \omega_0 \omega)^2}, \quad \frakb = \mu \rho_0 n_0 \frac{m^3 \omega^3 \delta \omega_0 }{(\omega_0^2-m^2\omega^2)^2 + (m \delta \omega_0 \omega)^2}. 
\end{equation}
 The variational form of \eqref{Helmholtz problems numerics} is given by
\begin{equation} \label{fem form}
	a(u, \phi) = l(\phi) \quad \text{for all } \phi \in H^1(\Omega; \C),
\end{equation}
where the sesquilinear form $a: \Hone \times \Hone \rightarrow \C$ is defined as 
\begin{equation}
	\begin{aligned}
a(u,\phi)\coloneqq &\,\begin{multlined}[t]		\intO \left(1+ \imath \omega \frac{m b}{c^2}\right)(\nabla u \cdot \overline{\nabla \phi}) \dx -(k^2+\fraka) \intO u \overline{\phi}\dx +\imath \frakb \intO u \overline{\phi}\dx 
\\	+ \int_{\partial \Omega} \left(1+\imath \omega\frac{m b}{c^2} \right)(\imath \omega m \beta + \gamma ) u \overline{\phi}\dG
\end{multlined}
	\end{aligned}
	\end{equation}
and $l(\phi) \coloneqq \intO f \overline{\phi}\dx$. We then have 
\begin{equation}
	\begin{aligned}
		\real \left(a(u,u)\right) = \|\nabla u\|^2_{\Ltwo} - (k^2+\fraka)\|u\|^2_{\Ltwo} + \left(\gamma - \frac{b\beta (m \omega)^2}{c^2}\right)\|u\|^2_{L^2(\partial \Omega)}.
	\end{aligned}
	\end{equation}
and
\begin{equation}
	\begin{aligned}
		\Im \left(a(u,u)\right) = \frac{m \omega b}{c^2} \|\nabla u\|^2_{\Ltwo} +\frakb \|u\|^2_{\Ltwo}+\left(\gamma \frac{b}{c^2}+\beta \right)\omega m  \|u\|^2_{L^2(\partial \Omega)}
	\end{aligned}
\end{equation}
Therefore, we can conclude that
\begin{equation}
	\begin{aligned}
		\sqrt{2} |a(u,u)| \geq &\,	\real \left(a(u,u)\right)+	\Im \left(a(u,u)\right) \\
		=&\,\begin{multlined}[t] \left(1 +\frac{m \omega b}{c^2}\right) \|\nabla u\|^2_{\Ltwo}+ (\frakb - k^2-\fraka)\|u\|^2_{\Ltwo}\\
		+	\left(\gamma \left(1+ \omega m \frac{b}{c^2}\right)+\beta\left(1 -\frac{b  m \omega}{c^2}\right) m \omega \right) \|u\|^2_{L^2(\partial \Omega)},
			\end{multlined}
	\end{aligned}
	\end{equation}
which showcases how the coercivity of $a(\cdot, \cdot)$ is influenced by the relative behavior of the various medium and frequency-dependent parameters. Note that from the complex part of \eqref{fem form}, we obtain the bound
\begin{equation}
	\frac{m \omega b}{c^2} \|\nabla u\|^2_{\Ltwo} +\frac{\frakb}{2} \|u\|^2_{\Ltwo}+\left(\gamma \frac{b}{c^2}+\beta \right)\omega m  \|u\|^2_{L^2(\partial \Omega)}  \leq  \frac{1}{2 \frakb}\|\fomega\|^2_{\Ltwo},
\end{equation}
which carries over to the finite element setting. 
\subsection{Numerical framework}
All simulations are performed using a conforming two-dimensional finite element discretization with linear Lagrange elements, using FEniCSx v0.9.0 (see, e.g., \cite{baratta_2023}) with Gmsh~\cite{gmsh_2009} employed for meshing. A direct LU factorization is used to solve the resulting linear systems without iterative refinement. \footnote{The program code for the simulations in this section is available as an ancillary file from the arXiv page of this paper.}   \\
\indent We use a circular domain of propagation $\Omega = B_{0.2}(0) \subset \mathbb{R}^2$ with a radius of $0.2 \, \si{\meter}$ and place a monopole source $h=h(x)$ at $x_0 = (0,0)^T$ defined as
\begin{align}
	h(x) = \begin{cases}
		\frac{a}{4 r_\delta} \left( 1 + \cos\left( \pi \frac{\left\| x - x_0 \right\|_{l^2} }{2 r_{\delta}}\right) \right)  \qquad \quad & \left\| x - x_0 \right\|_{l^2} \leq 2 r_{\delta}, \\
		0 & \text{otherwise.}
	\end{cases}
\end{align}
We present numerical results for various values of $a$, using $\omega=\omega_0$ and $r_{\delta}=0.004$ unless stated otherwise. The physical parameters used in the simulations are chosen as typical in ultrasound contrast imaging (see, e.g.,~\cite{hoff2001acoustic, hamilton1998nonlinear}), and are listed in Table \ref{table: parameters}.

\begin{table}[h]
	\begin{center}
		\begin{tabular}{l l l | l l l}
			\hline
			&&&&&\\[-0.25cm]
			$c$ & speed of sound & $1500\,\si{\m / \s}$ & $b$ & diffusivity of sound  & $1 \cdot 10^{-3}$\\
			$\rho_0$ & mixture mass density & $1000 \, \si{\kilogram / \m^3}$ & $\beta_a$ & nonlinearity coefficient & $ 3.5$ \\
			$R_0$ & initial radius  & $2 \, \si{\micro \m}$& $n_0$ & bubble number density & $1 \cdot 10^{12} \, \si{1 / \milli \liter}$ \\
			$P_0$ & vapor pressure & $100 \, \si{\pascal}$ & $\kappa$  & adiabatic exponent & $1.4$ 
			 \\
			$\nu$ & kinematic viscosity & $8.9 \cdot 10^{-6} \si{\square \meter / \second}$ &
			&&\\[0.2cm]
			%& & \\
			\hline
		\end{tabular}
	\end{center}
	\caption{Overview of the parameter values used in the simulations.}
	\vspace{-0.0cm}
	\label{table: parameters}
\end{table}
\noindent Using the values in Table~\ref{table: parameters}, we compute
\begin{equation}
	\omega_0 = \sqrt{\frac{3 \kappa P_0}{\rho_0 R_0^2}} \approx 0.324 \, \si{\mega \hertz}
\end{equation}
and additionally $\delta = \frac{4 \nu}{\omega_0 R_0^2}$, $v_0 = \frac{4 \pi}{3} R_0^3$, $\mu = \frac{4 \pi R_0}{\rho_0}$, $\ca= \frac{(\kappa + 1) \omega_0^2}{2 v_0}$, and $\cb = \frac{1}{6 v_0}$ are fixed. Furthermore, we set $\beta = \frac{1}{c}$ and $\gamma = 1$ in the boundary conditions for Helmholtz problems.

\subsection{The influence of spatial discretization}\label{sec: numerical evaluation} 
 
For a fixed $N$, we first investigate the convergence of 
a quantity of interest given by 
\[\left\| \Re(p_{\hFEM}^{N}) \right\|_{L^{\infty}(0,T; L^2(\Omega))},\] where $p_{\hFEM}^N$ denotes the approximate pressure, as the spatial discretization parameter $\hFEM \searrow 0$.
A common rule of thumb for finite element discretizations of the classical Helmholtz equation with linear elements is to use at least $10$ elements per wavelength in each spatial direction. 
If the mesh is not refined accordingly, one encounters the above-mentioned pollution effect, where numerical phase errors grow with increasing $k$; see, e.g., \cite{babuska2000, ihlenburg1995finite, melenk2010}. 

\begin{figure}[h]
	\centering
	\includegraphics[width=0.75\textwidth]{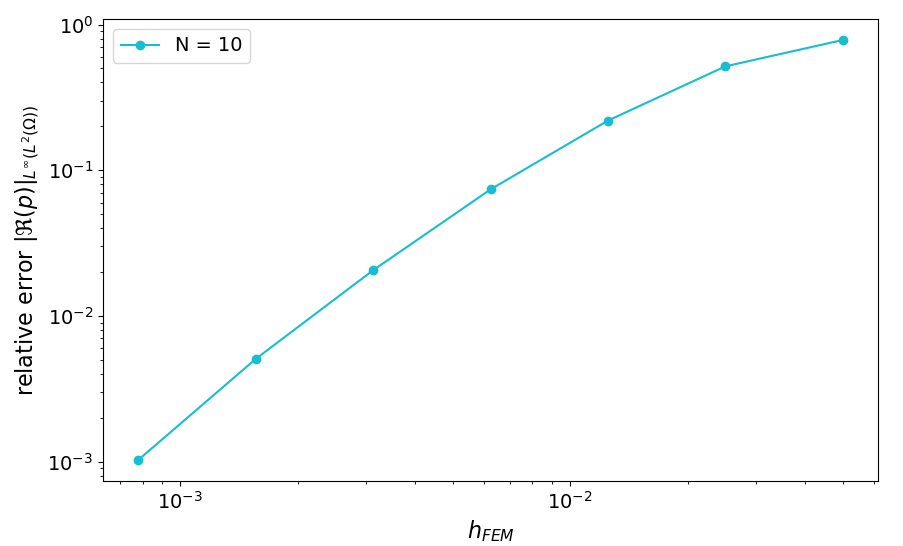}
	\caption{
		Relative error  \eqref{eq:relative_difference} of the quantity of interest
		%reconstructed pressure field
		 with respect to the mesh size $\hFEM$ on a log--log scale.
		 %, using the quantity of interest on the finest spatial mesh as reference.
	}
	
	\label{fig:conv_h}
\end{figure}

Figure~\ref{fig:conv_h} illustrates the behavior of the relative difference
\begin{equation} \label{eq:relative_difference}
	\frac{
			\left|
			\bigl\|\Re(p_{\hFEM}^{N})\bigr\|_{L^\infty(0,T;L^2(\Omega))}
			-
			\bigl\|\Re(p^{\mathrm{ref}})\bigr\|_{L^\infty(0,T;L^2(\Omega))}
			\right|
		}{
			\bigl\|\Re(p^{\mathrm{ref}})\bigr\|_{L^\infty(0,T;L^2(\Omega))}
		},
\end{equation}
as the mesh is refined, using the solution on the finest spatial mesh as reference. 
A clear algebraic decay is observed for decreasing mesh size~\change{$h_{\textup{FEM}}$}, indicating that the spatial discretization error is well controlled in the considered parameter regime and that the chosen mesh resolutions are sufficient to avoid significant pollution effects for the frequencies under consideration. The similarity of the convergence curves for different truncation levels $N$ shows that the spatial discretization error dominates over the effect of the multiharmonic truncation.

\begin{table}[h]
	\centering
	\begin{tabular}{cccc}
		\toprule
		$\mh$ & Nodes & Elements & Time \\
		\midrule
		0.000390625   & 9,564,795  & 1,933,589  & 9984.8 s \\
		0.00078125 & 239,024    & 478,047    & 1244.3 s \\
		0.0015625  & 60,136     & 120,271    & 145.5 s \\
		0.003125  & 15,218     & 30,435     & 19.8 s \\
		0.00625  & 3,896      & 7,791      & 5.16 s \\
		0.0125    & 1,010      & 2,019       & 2.72 s \\
		0.025     & 279        & 557       & 2.24 s \\
		0.05     & 85        & 169       & 2.15 s \\
		\bottomrule
	\end{tabular}
		\captionsetup{width=0.6\textwidth}
	\caption{Mesh sizes, degrees of freedom, and computational time for different spatial discretizations.}
	\label{tab:data}
\end{table}
Additional quantitative data are provided in Table~\ref{tab:data}, which lists the corresponding mesh sizes, number of nodes and elements, as well as the computational time required for each case. This highlights the significant increase in computational cost associated with finer meshes, especially for very small $\mh$, where both the number of degrees of freedom and the runtime grow substantially. All simulations were carried out on a standard laptop (Intel Core i7, 16GB RAM, no dedicated GPU), underscoring the feasibility of the method even without access to high-performance computing resources.

In the following simulations, we use a mesh size of $\mh = 0.003$.  
Unless stated otherwise in the figure captions, we set $r_{\delta} = 0.004$, $a = 10^5$, and $\omega = \omega_0$. 
\subsection{Comparison of real- and complex-valued fields}\label{sec: numerics multiharmonic}
We compare the five-harmonic pressure obtained with the linearized multiharmonic cut-off algorithm using the complex-valued formulation \eqref{system_v_0_0} and the real-valued formulation \eqref{eq: multi_sim1}. In the considered numerical setting, no noticeable differences between the resulting pressure fields are observed.

\begin{figure}[h]
	\centering
	\includegraphics[width = 0.67\linewidth]{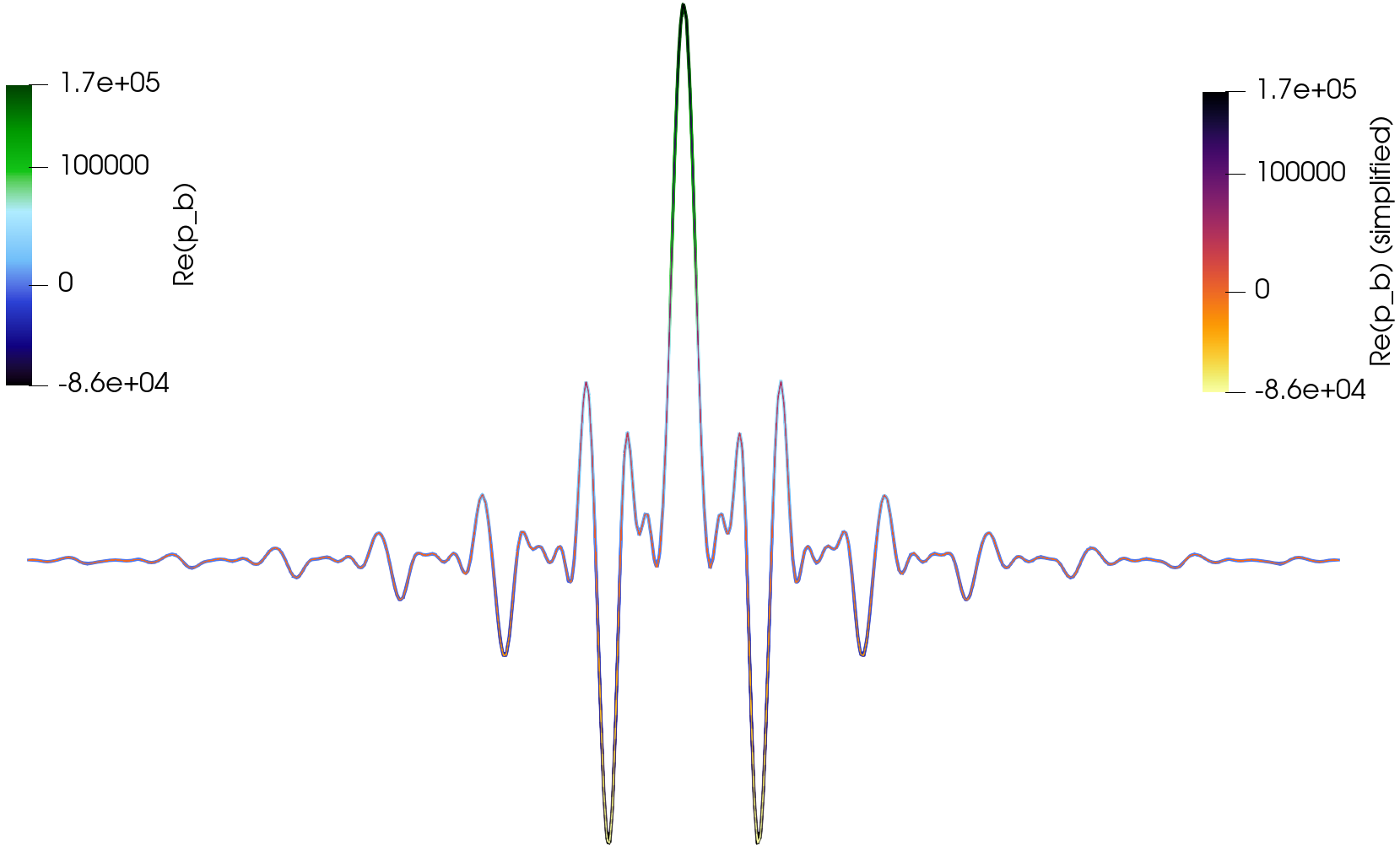}
	\caption{Five-harmonic expansion $\real(p_{\hFEM}^{(5)}(x_1,x_{2,0},t_0))$ plotted as a function of $x_1$ for a fixed $x_2=x_{2,0}$ and $t=t_0$ resulting from complex fields in \eqref{system_v_0_0} vs. real setting in \eqref{eq: multi_sim1}.}
	\label{fig: simplified_vs_real}
\end{figure}
	
To further quantify this observation, we compare the $L^2(\Omega)$-norms of the harmonic components obtained from the complex- and real-valued formulations. The relative differences in these norms are below $5 \cdot 10^{-3}$ for the dominant harmonics $m = 1,\ldots,4$, and below $2 \cdot 10^{-2}$ for higher harmonics $m \ge 5$. At the same time, the amplitudes of these higher harmonics are already several orders of magnitude smaller than that of the fundamental mode.

We therefore consider both formulations to be numerically equivalent in the considered parameter regime. Since the complex formulation leads to a computationally more efficient algorithm, it is used in all subsequent simulations.

\subsection{The influence of the number of harmonics}

We next analyze how many harmonics need to be retained in the multiharmonic expansion for the considered parameter settings. We first investigate the \change{behavior} of the \change{approximate} pressure field with respect to the truncation level $N$ \change{for a fixed $h_{\textup{FEM}}$}, which motivates the choice of a suitable reference solution.

\begin{figure}[h]
	\centering
	\includegraphics[width=0.67\textwidth]{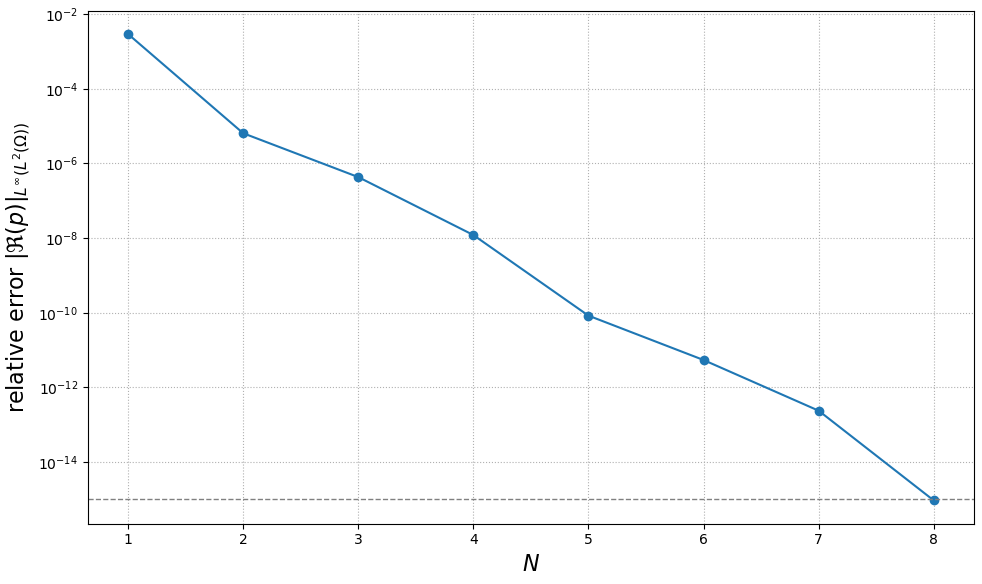}
	\caption{
		Relative error \eqref{eq:relative_difference} of the quantity of interest 
	with respect to the truncation level $N$ on a semi-log scale (logarithmic $y$-axis)
	}
	\label{fig:conv_N}
\end{figure}

Figure~\ref{fig:conv_N} shows the relative error of the  pressure field \change{as defined in \eqref{eq:relative_difference} for increasing truncation levels $N$}, using the solution with $N=10$ harmonics as reference. For larger values of $N$, the relative error levels off at approximately machine precision, reflecting that further changes in the norm are below the accuracy of the numerical discretization and floating-point arithmetic. \\
\indent To illustrate how the contributions from individual harmonics are reflected in the approximate pressure field, we additionally consider the effect of truncating the multiharmonic expansion in the time domain. For a fixed reference time $t_0$, we compute the pointwise differences 
\[
\left|p_{\hFEM}^N(x,t_0) -p_{\hFEM}^{(10)}(x,t_0) \right|,
\]
where $p_{\hFEM}^N(x,t)$ denotes the approximate pressure field obtained using the first $N$ harmonics and the solution with $N=10$ harmonics serves as reference. 
\begin{figure}[h]
	\centering
	\begin{minipage}[t]{0.24\textwidth}
		\centering
		\includegraphics[width=\textwidth]{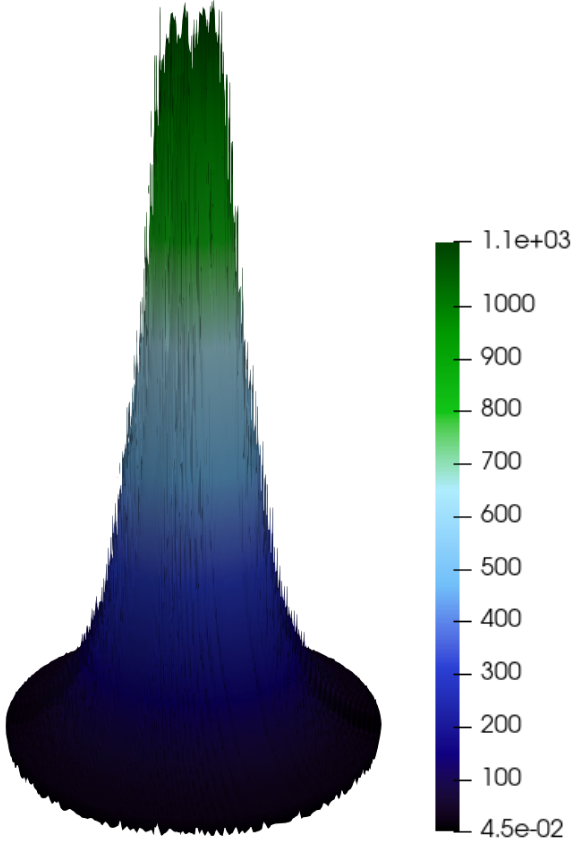}
		\caption*{$N=2$}
	\end{minipage}\hfill
	\begin{minipage}[t]{0.24\textwidth}
		\centering
		\includegraphics[width=\textwidth]{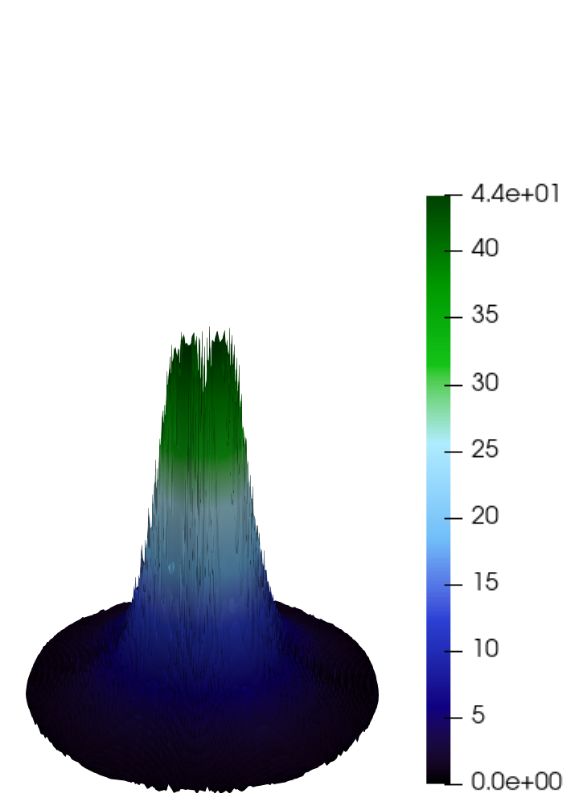}
		\caption*{$N=3$}
	\end{minipage}\hfill
	\begin{minipage}[t]{0.24\textwidth}
		\centering
		\includegraphics[width=\textwidth]{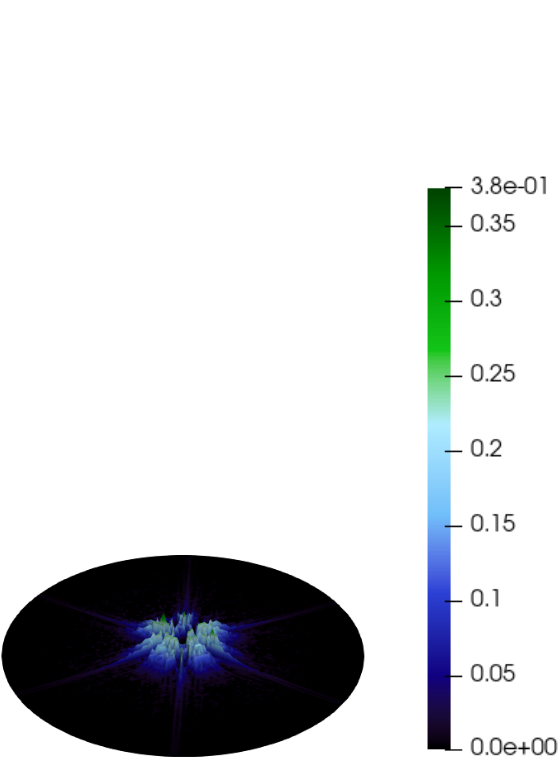}
		\caption*{$N=5$}
	\end{minipage}\hfill
	\begin{minipage}[t]{0.24\textwidth}
		\centering
		\includegraphics[width=\textwidth]{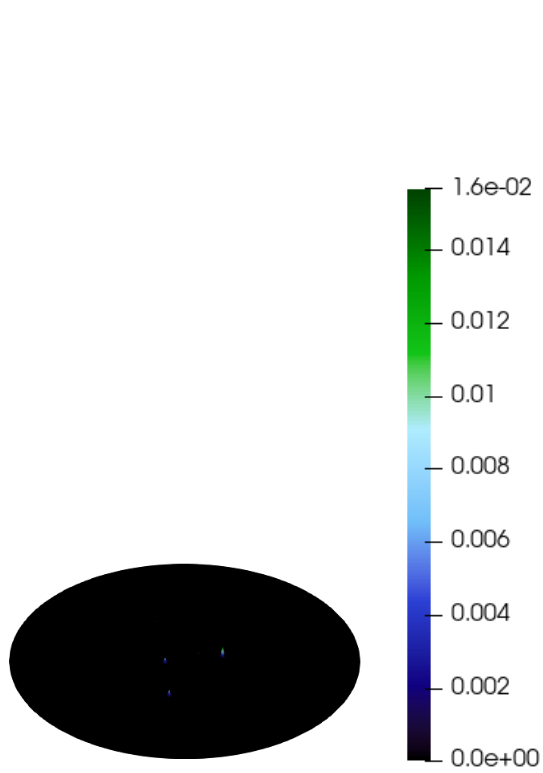}
		\caption*{$N=7$}
	\end{minipage}
	\caption{
		Pointwise differences $\lvert p_{\hFEM}^N(x,t_0) - p_{\hFEM}^{(10)}(x,t_0) \rvert$ for $a=10^5$ at a fixed time $t_0$ and different truncation levels $N=2,3,5,7$.
	}
	\label{fig:diff_A1e5}
\end{figure}

Figure~\ref{fig:diff_A1e5} visualizes the spatial structure of the truncation error for different truncation levels $N$ in the strongly nonlinear case $a=10^5$. For small values of $N$, the deviations from the reference solution are large and extend over significant parts of the domain, whereas increasing the number of retained harmonics leads to a rapid reduction and spatial localization of the differences. In particular, once approximately five harmonics are included, the remaining discrepancies become small throughout the domain, which is fully consistent with the harmonic-wise decay observed in Table~\ref{tab:relative_decay}. \\
\indent We now fix the truncation index to $N=10$ and examine the harmonic content of the \change{numerical} solution in more detail. To this end, we compute the $L^2(\Omega)$-norms $\|p^N_{\hFEM, m}\|_{L^2(\Omega)}$ %\vanja{I think here it might be better to use $p^N_{\hFEM, m}$}
of the harmonic coefficients and normalize them by the fundamental mode,
\[
r_m(a) := \frac{\|p^N_{\hFEM, m}\|_{L^2(\Omega)}}{\|p^N_{\hFEM,1}\|_{L^2(\Omega)}},
\]
for each driving amplitude $a$ and harmonic index $m$. \\
\indent From Table~\ref{tab:relative_decay} we observe that the relative amplitudes of higher harmonics decrease rapidly with increasing $m$, and that this decay becomes more pronounced as the driving amplitude $a$ is reduced, indicating a weaker nonlinear response. For $a=10^3$, the dominant contribution comes from the first two harmonics, while for $a=10^4$ the first three harmonics contribute noticeably. In the strongly nonlinear case $a=10^5$, appreciable contributions extend to higher harmonic indices, indicating that retaining at most four to five harmonics is appropriate to accurately represent the pressure field.

\begin{table}[h]
	\centering
	\begin{tabular}{c|cccc}
		\toprule
		$a$ &
		$r_2(a)$ &
		$r_3(a)$ &
		$r_4(a)$ &
		$r_5(a)$ \\
		\midrule
		$10^3$ & $2.23\cdot 10^{-4}$ & $8.11\cdot 10^{-8}$ & $2.94\cdot 10^{-11}$ & $1.15\cdot 10^{-14}$ \\
		$10^4$ & $2.23\cdot 10^{-3}$ & $8.11\cdot 10^{-6}$ & $2.94\cdot 10^{-8}$ & $1.15\cdot 10^{-10}$ \\
		$10^5$ & $2.23\cdot 10^{-2}$ & $8.11\cdot 10^{-4}$ & $2.94\cdot 10^{-5}$ & $1.15\cdot 10^{-6}$ \\
		\bottomrule
	\end{tabular}
	\caption{
		Relative $L^2(\Omega)$-norms of higher harmonics
		$r_m(a)$ for $m=2,\dots,5$ and $N=10$.
	}
	\label{tab:relative_decay}
\end{table}

\subsection{Comparison of the pressure field with and without bubbles}
Finally, we compare the pressure field obtained from the multiharmonic formulation in the presence and absence of bubble dynamics in order to assess how bubble coupling affects the nonlinear response of the system. We examine the pressure field in the time domain and the harmonic coefficients.

\begin{figure}[h]
	\centering
	\begin{subfigure}[t]{0.48\textwidth}
		\centering
		\includegraphics[width=0.78\textwidth]{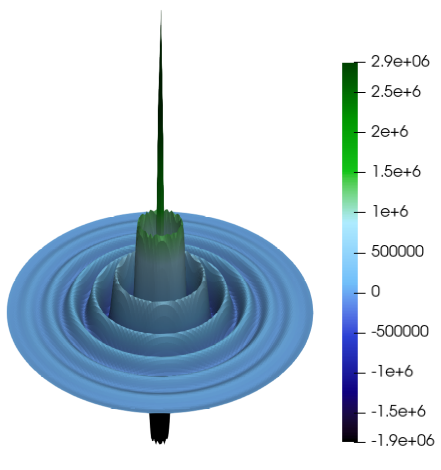}
		\caption{Pressure with bubbles in the medium.}
	\end{subfigure}\hfill
	\begin{subfigure}[t]{0.48\textwidth}
		\centering
		\includegraphics[width=0.78\textwidth]{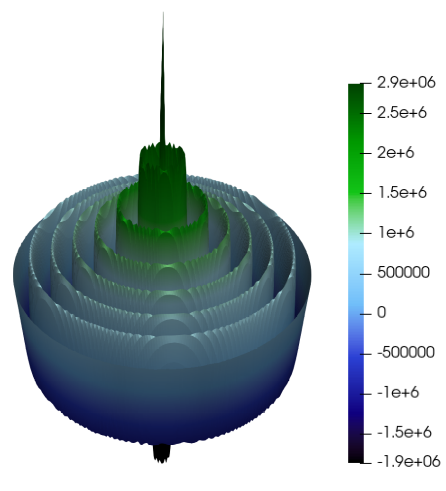}
		\caption{Pressure without bubbles in the medium.}
	\end{subfigure}
	\caption{
		Pressure field $\Re(p_{\hFEM}^{(5)}(x,t_0))$ at a fixed reference time $t_0$ for $a=10^5$, with and without bubbles in the medium.
	}
	\label{fig:bub_vs_nobub}
\end{figure}

Figure~\ref{fig:bub_vs_nobub} compares the real part of the pressure distribution within the domain at a specific time point with and without the presence of bubbles in the medium; that is, the real part of the pressure obtained using the algorithm in \eqref{two-harmonic: nonlinear} and the same algorithm with $n_0=\tilde{\beta}=0$. The presence of bubbles leads to an overall damping effect, resulting in a reduced pressure amplitude across the domain. This observation aligns with the behavior predicted by the system of equations in %\eqref{two-harmonic: linear} and 
\eqref{two-harmonic: nonlinear}, as the modified speed of sound introduces attenuation. Although the source amplitude is fixed in this case, we note that the differences become more pronounced for larger source amplitudes, leading to higher overall pressure levels.

To quantify these differences, we compare the harmonic content of the pressure field in both settings. For a fixed truncation level $N=10$, we consider the $L^2(\Omega)$-norms of the pressure harmonics $\|p_{\hFEM,m}^N\|_{L^2(\Omega)}$ and report their relative contributions with respect to the fundamental mode, 
\[
r_m := \frac{\|p_{\hFEM,m}^N\|_{L^2(\Omega)}}{\|p_{\hFEM,1}^N\|_{L^2(\Omega)}}.
\]

Table~\ref{tab:bub_vs_nobub} shows that higher pressure harmonics have systematically larger relative amplitudes in the absence of bubbles. In particular, the decay of the relative harmonic ratios is significantly slower in the bubble-free case, with comparable contributions persisting over a broader range of harmonic indices. When bubble dynamics are included, the relative amplitudes of higher pressure harmonics decrease much more rapidly, dropping below $10^{-6}$ already around the fifth harmonic. In addition, the harmonic coefficients associated with the bubble variable $v_m$ are several orders of magnitude smaller than the corresponding pressure harmonics for all $m$, indicating that nonlinear effects transferred into the bubble dynamics are strongly damped.
\begin{table}[h]
	\centering
	\begin{tabular}{c|ccccc}
		\toprule
		& $r_2$ & $r_3$ & $r_4$ & $r_5$ & $r_6$ \\
		\midrule
		with bubbles 
		& $2.23\cdot10^{-2}$ & $8.11\cdot10^{-4}$ & $2.94\cdot10^{-5}$ 
		& $1.15\cdot10^{-6}$ & $1.13\cdot10^{-7}$ \\
		without bubbles 
		& $5.48\cdot10^{-2}$ & $1.78\cdot10^{-3}$ & $5.50\cdot10^{-5}$ 
		& $2.61\cdot10^{-6}$ & $2.64\cdot10^{-6}$ \\
		\bottomrule
	\end{tabular}
	\caption{
		Relative $L^2(\Omega)$-norms $r_m$ of the pressure harmonics for $a = 10^5$ with and without bubbles in the medium ($N=10$).
	}
	\label{tab:bub_vs_nobub}
\end{table}

We further examine the temporal evolution of the five-harmonic expansion at different spatial points in this context.

\begin{figure}[h] 
	\centering
	\begin{subfigure}{0.32\textwidth}
		\centering
		\includegraphics[width=\linewidth]{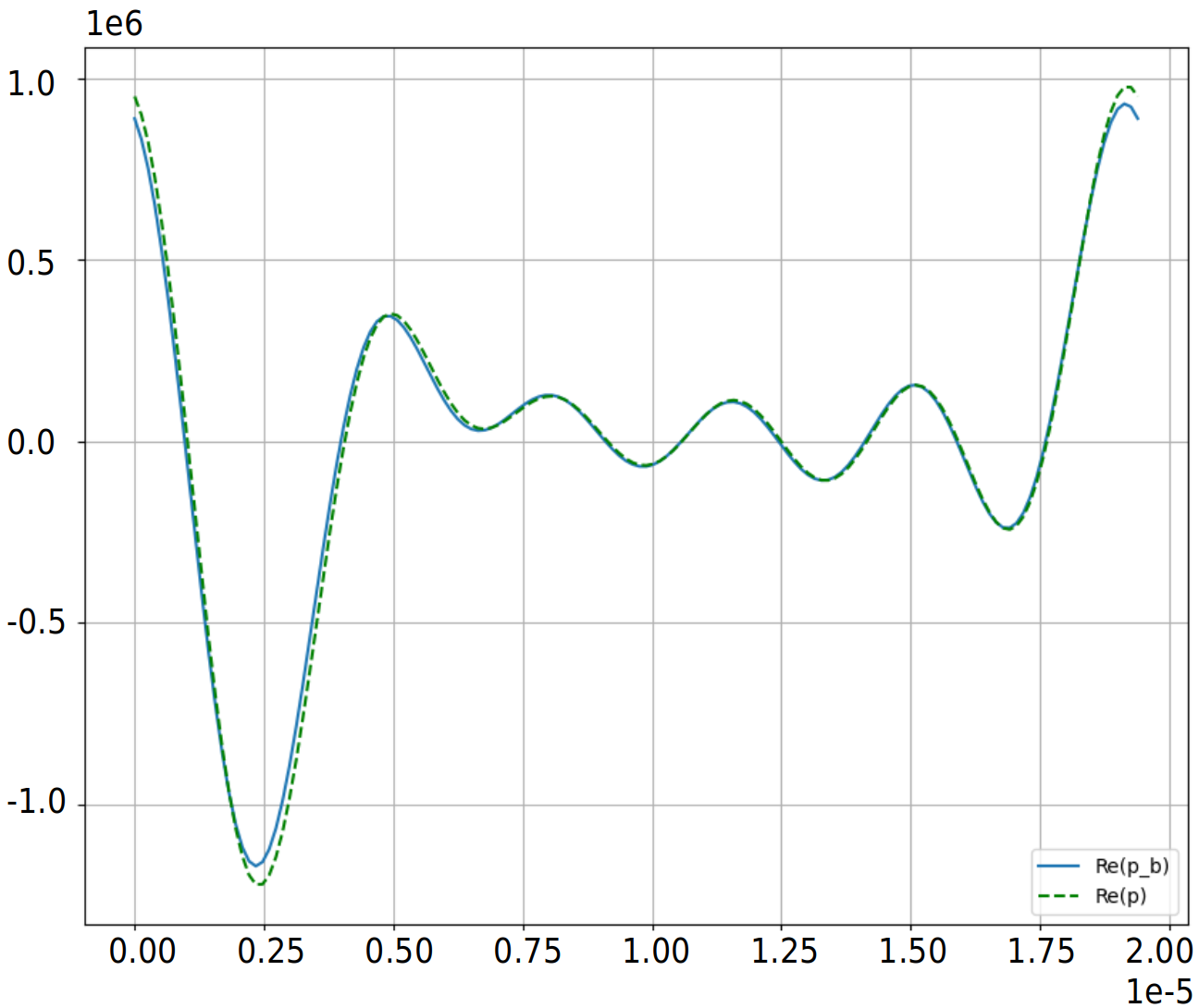}
		\caption{$\real({p_{\hFEM}^{(5)}((0,0),t)})$}
		\label{fig:image1}
	\end{subfigure}
	\hfill
	\begin{subfigure}{0.32\textwidth}
		\centering
		\includegraphics[width=\linewidth]{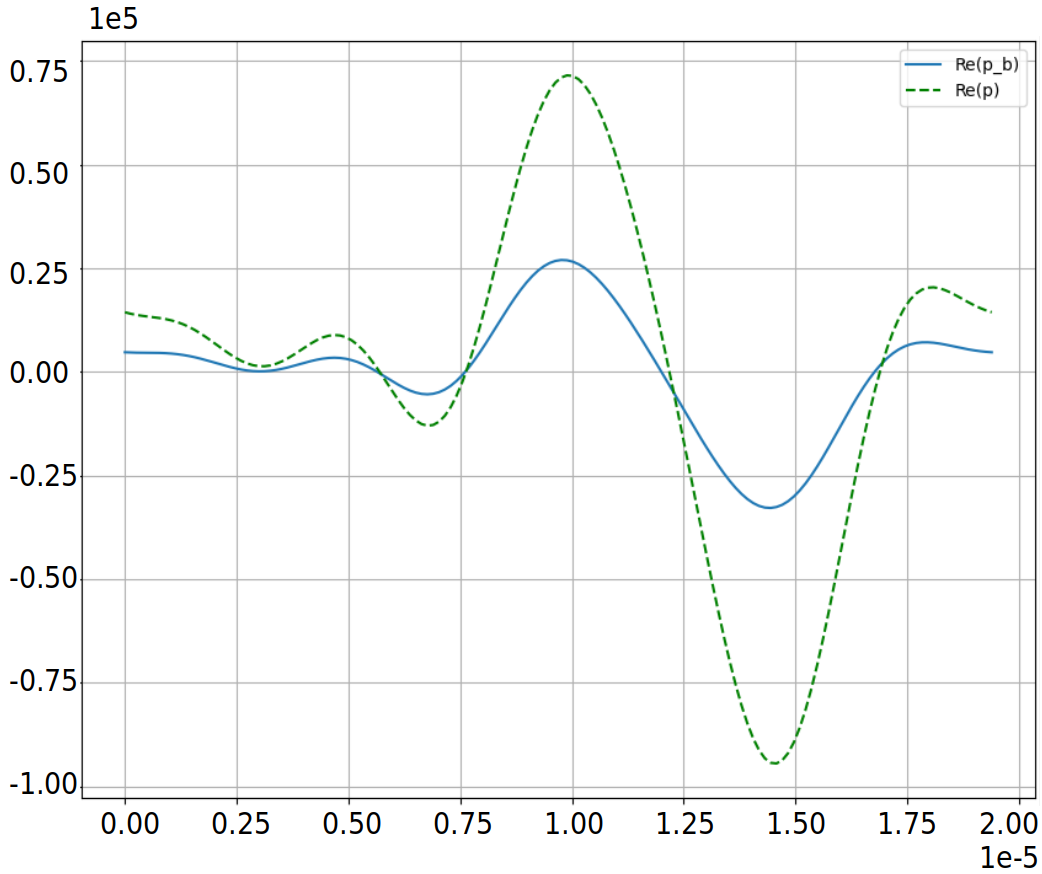}
		\caption{$\real({p_{\hFEM}^{(5)}((0,0.05),t)})$}
		\label{fig:image2}
	\end{subfigure}
	\hfill
	\begin{subfigure}{0.32\textwidth}
		\centering
		\includegraphics[width=\linewidth]{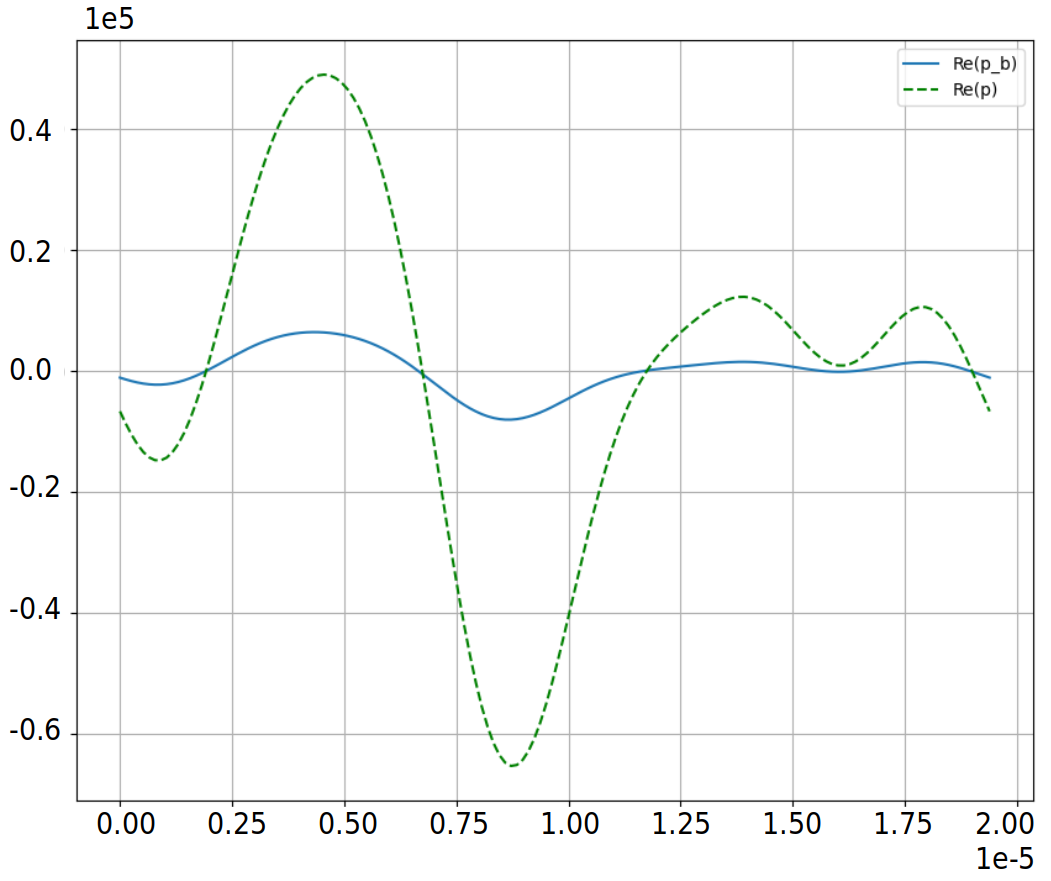}
		\caption{$\real({p_{\hFEM}^{(5)}((0,0.1),t)})$}
		\label{fig:image2}
	\end{subfigure}
	\vspace{-0.2cm}
	
	\caption{Real part of the pressure from the five-harmonic expansion computed using \eqref{system_v_0_0}  
		without bubbles ($n_0=0$, dashed green line) and with bubbles (blue line) over time at three spatial points.}
	\label{fig: 5harm_over_time}
\end{figure} \vspace*{-4mm}
\begin{figure}[h] 
	\centering
	\includegraphics[width=0.4\linewidth]{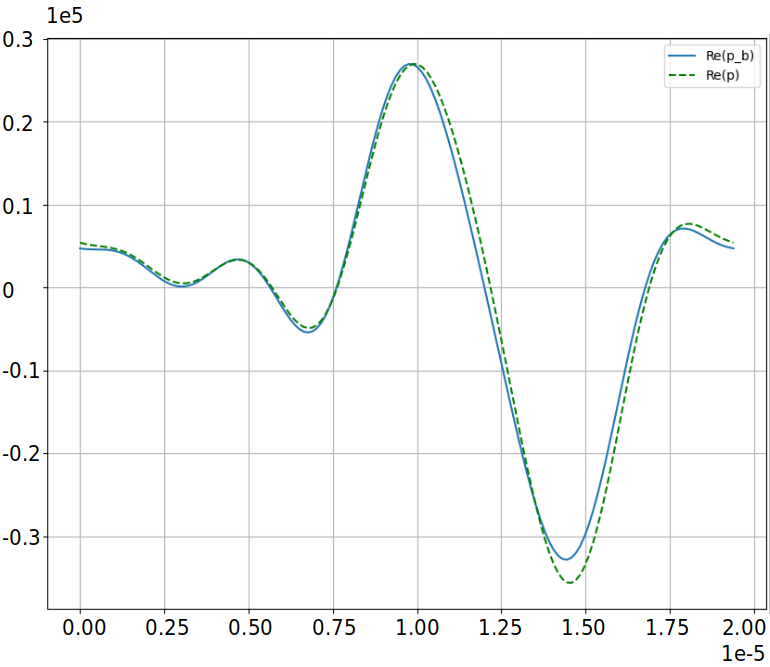}
	\caption{Real part of the pressure $\real({p_{\hFEM}^{(5)}(0,0.05,t)})$ from the five-harmonic  expansion computed using \eqref{system_v_0_0}  without bubbles ($n_0=0$, dashed green line) and with bubbles (blue line, scaled to the maximum of the green curve) over time.}
	\label{fig: zoom_scaled_5harm}
\end{figure}
As illustrated in Figure \ref{fig: 5harm_over_time}, the presence of microbubbles near the source has little effect on the pressure waves. However, as the distance from the source increases, the signal strength decreases. The microbubbles not only attenuate the waveform but also enhance its nonlinear characteristics. To better visualize these effects, we scale the pressure waveform obtained from the five-harmonic expansion with bubbles to the maximum of the waveform without microbubbles, as shown in Figure \ref{fig: zoom_scaled_5harm}.

\subsection{Conclusion}

The numerical results demonstrate that microbubbles introduce significant attenuation and phase shifts to the wave propagation, particularly at greater distances from the source. The impact increases with microbubble density and source amplitude, and is further influenced by the driving frequency. 
Regarding the number of harmonics $N$, in the considered numerical settings using three to five harmonics within the simplified multiharmonic framework obtained from complex fields provides a good compromise between computational efficiency and accuracy. This relatively low number of harmonics makes the approach promising for use in practical applications.

\section*{Outlook}

In practice, nonlinear acoustic interactions between microbubbles and ultrasound waves generate not only harmonics, which are frequency components at integer multiples of the driving frequency, but also subharmonics, which appear at fractional multiples of the driving frequency, such as $\frac{\omega}{2}$, $\frac{\omega}{3}, \, \ldots$; see, e.g.,~\cite{krishna1999subharmonic}. Subharmonics primarily appear due to non-spherical deformations and multibubble interactions. While these effects are not included in our current model, extending the framework to incorporate them would be an interesting direction for future research. \\
\indent Expanding the present theoretical and numerical framework to explore  other phenomena relevant to applications of focused ultrasound waves is also of practical interest. For instance,  localized heating, cavitation, and nonlocal attenuation play an important role in therapeutic ultrasound applications such as targeted drug delivery, and it would be worthwhile to investigate the potential role of multiharmonic expansions in these modeling contexts. 

Furthermore,  investigating inverse problems related to reconstructing spatially varying parameters, such as the bubble number density $n_0=n_0(x)$ or the nonlinearity parameter $\beta_a=\beta_a(x)$ from measured acoustic signals is an important task in the context of contrast-enhanced ultrasound imaging,  as it could help improve diagnostic accuracy in the long run.

\subsection*{Acknowledgments} We are thankful to Prof.\ Barbara Kaltenbacher (University of Klagenfurt) for her valuable comments. \teresa{We also thank the reviewer for their careful reading and helpful comments, which have led to an improved version of the manuscript.} This research was funded in part by the Austrian Science Fund (FWF) [10.55776/DOC78]. The work of V.N.  was partially supported by the Dutch Research Council (NWO) under the grant OCENW.M.23.371 \sloppy with Grant ID \href{https://doi.org/10.61686/VLHHB85047 }{https://doi.org/10.61686/VLHHB85047}.

For open-access purposes, the authors have applied a CC BY public copyright license to any author-accepted manuscript version arising from this submission.

\appendix
	\section{Proof of Lemma~\ref{lemma: ode}} \label{proof Lemma}
We present here the postponed proof of Lemma~\ref{lemma: ode}.	
\begin{proof}
	We can write the linearized second-order ODE in the matrix system form with $V=[v \ \vt]^T$:
	\begin{equation} \label{matrix ode}
		\begin{aligned}
			&V_t =\, AV +  F, \quad A= \begin{bmatrix}
				0 & 1 \\ -\omega_0^2 & -\delta \omega_0
			\end{bmatrix},\ F=\begin{bmatrix}
				0 \\ f
			\end{bmatrix}
		\end{aligned}
	\end{equation}
	a.e. in space.	Let $\vone$ and $\vtwo$ be to linearly independent solutions of the homogeneous equation $ \vtt + \delta \omega_0 \vt + \omega_0^2 v =0$. 
	\begin{equation}
		\begin{aligned}
			(\vone, \vtwo) = \begin{cases}
				(e^{r_1 t}, e^{r_2 t}), \ r_{1,2} = \frac{- \delta \omega_0\pm \sqrt{\delta^2 \omega_0^2 - 4 \omega_0^2}}{2} \quad &\text{if} \quad \delta \omega_0 > 2 \omega_0, \\
				(e^{-\omega_0 t}, t e^{-\omega_0 t}), \   \quad &\text{if} \quad \delta \omega_0 = 2 \omega_0, \\
				e^{- \delta \omega_0 t/2}(\cos \alpha t,\, \sin \alpha t), \ \alpha = \frac{ \sqrt{ 4 \omega_0^2-\delta^2 \omega_0^2}}{2}  \quad &\text{if} \quad \delta \omega_0 < 2 \omega_0.
			\end{cases}
		\end{aligned}
	\end{equation}
	~
	
	\noindent (i) Let first $\delta \omega_0 > 2 \omega_0$. Then the fundamental matrix is given by
	\begin{equation}
		\begin{aligned}
			Z(t) 			=\, \begin{bmatrix}
				\vone & \vtwo \\ (\vone)' \quad & (\vtwo)'
			\end{bmatrix}=\, \begin{bmatrix}
				e^{r_1 t} & e^{r_2 t} \\  r_1 e^{r_1 t} & r_2 e^{r_2 t}
			\end{bmatrix}, 
		\end{aligned}
	\end{equation}
	and we have
	\begin{equation}
		\begin{aligned}
			\left(Z(t) \right)^{-1} = \frac{1}{r_2-r_1}e^{-(r_1+r_2)t} \begin{bmatrix}
				r_2 e^{r_2 t} & -e^{r_2 t} \\ -r_1 e^{r_1 t} \ & e^{r_1 t}
			\end{bmatrix},
		\end{aligned}
	\end{equation}
	where $r_1-r_2 = \sqrt{\delta^2 \omega_0^2-4 \omega_0^2}$ and $r_1+r_2 = -\delta \omega_0$. Then the principal fundamental matrix is
	\begin{equation}
		\begin{aligned}
			X(t) = Z(t) \left(Z(0)\right)^{-1} =&\, \frac{1}{r_2-r_1} \begin{bmatrix}
				e^{r_1 t} & e^{r_2 t} \\  r_1 e^{r_1 t} & r_2 e^{r_2 t}
			\end{bmatrix}  \begin{bmatrix}
				r_2  & -1 \\ -r_1  \ & 1
			\end{bmatrix} \\
			=&\, \frac{1}{r_2-r_1}  \begin{bmatrix}			
				r_2 e^{r_1t}-r_1 e^{r_2 t} & -e^{r_1t}+e^{r_2t}\\ r_1 r_2 (e^{r_1 t} -e^{r_2t}) & -r_1 e^{r_1t} +r_2 e^{r_2 t}
			\end{bmatrix}.
		\end{aligned}
	\end{equation}
	Note that $X_t =A X$.	Furthermore,
	\begin{equation}
		\begin{aligned}
			(X(t))^{-1} = 	 \frac{1}{r_2-r_1}    \begin{bmatrix}
				r_2 e^{-r_1 t}-r_1 e^{-r_2t} & 	-e^{-r_1t}+e^{-r_2t} \\ r_1 r_2 (e^{-r_1 t} - e^{-r_2t}) & 	-r_1 e^{-r_1t}+r_2 e^{-r_2 t}
			\end{bmatrix}.
		\end{aligned}
	\end{equation}
	Then using the method of  variation of parameters, we find that the solution of \eqref{matrix ode} has the form
	\begin{equation} \label{sol}
		\begin{aligned}
			V(t) = X(t)\left(V(0)+ \intt \left(X(\tau)\right)^{-1}f(\tau)\dtau \right).
		\end{aligned}
	\end{equation}
	Since periodic solutions must satisfy $V(0)=V(T)$, we then conclude that
	\begin{equation}
		\begin{aligned}
			(I-X(T))	V(0)= X(T) \int_0^T \left(X(\tau)\right)^{-1}f(\tau)\dtau.
		\end{aligned}
	\end{equation}
	We have
	\begin{equation}
		\det (I- X(T)) = 1 - (e^{r_1 T}+ e^{r_2 T}) + e^{(r_1 + r_2)T} = (1-e^{r_1 T })(1-e^{r_2 T })\neq 0.
	\end{equation}
	Therefore, since the matrix $I-X(T)$ is invertible, we find that
	\begin{equation} \label{ic}
		\begin{aligned}
			V(0)= (I-X(T))^{-1}X(T) \int_0^T \left(X(\tau)\right)^{-1}f(\tau)\dtau.
		\end{aligned}
	\end{equation}
	Furthermore, from \eqref{sol} and \eqref{ic}, we have the following bound:
	\begin{equation}
		\begin{aligned}
			\|V\|_{L^\infty(\Linf)}+\|V_t\|_{L^2(\Linf)} \lesssim \|f\|_{L^2(\Linf)}.
		\end{aligned}
	\end{equation}	
	~
	
	\noindent {(ii)} Let next $\delta \omega_0 = 2 \omega_0$. Then the fundamental matrix is given by
	\begin{equation}
		\begin{aligned}
			Z(t) =e^{-\omega_0 t} \begin{bmatrix}
				1 & t \\ -\omega_0 \ & 1-\omega_0 t
			\end{bmatrix}, \quad  \left(Z(t) \right)^{-1} = e^{\omega_0 t} \begin{bmatrix}
				1- \omega_0 t & -t \\ \omega_0 \ & 1
			\end{bmatrix}.
		\end{aligned}
	\end{equation}
	Then
	\begin{equation}
		\begin{aligned}
			X(t) = e^{-\omega_0 t} \begin{bmatrix}
				1 & t \\ -\omega_0 \ & 1-\omega_0 t
			\end{bmatrix} \begin{bmatrix}
				1 & 0 \\ \omega_0 \ & 1
			\end{bmatrix} =  e^{-\omega_0 t} \begin{bmatrix}
				1+ t\omega_0 \ & t \\ - \omega_0^2 t \ & 1- \omega_0 t
			\end{bmatrix}
		\end{aligned}
	\end{equation}
	and
	\begin{equation}
		\begin{aligned}
			(X(t))^{-1} = 	e^{\omega_0 t} \begin{bmatrix}
				1- \omega_0 t & -t \\ \omega_0^2 t \ & 1+ \omega_0 t
			\end{bmatrix}.
		\end{aligned}
	\end{equation}
	Then using the method of  variation of parameters, we find that the solution has the form
	\begin{equation}
		\begin{aligned}
			V(t) = X(t)\left(V(0)+ \intt \left(X(\tau)\right)^{-1}f(\tau)\dtau \right).
		\end{aligned}
	\end{equation}
	Since periodic solutions must satisfy $V(0)=V(T)$, we then have
	\begin{equation}
		\begin{aligned}
			(I-X(T))	V(0)= X(T) \int_0^T \left(X(\tau)\right)^{-1}f(\tau)\dtau
		\end{aligned}
	\end{equation}
	and 
	\begin{equation}
		\det (I- X(T)) = (1-e^{-\omega_0 T})^2 > 0,
	\end{equation}
	so we can reason similarly to the first case. \vspace*{2mm}
	
	\noindent{(iii)}: Lastly, let $\delta \in (0, 2)$. The fundamental matrix is then given by 
	\begin{equation}
		\begin{aligned}
			Z(t) 
			=&\, e^{-\delta \omega_0 t/2} \begin{bmatrix}
				\cos \alpha t & \sin \alpha t \\ - \frac{\delta \omega_0}{2} \cos \alpha t - \alpha \sin \alpha t \quad & -\frac{\delta \omega_0}{2} \sin \alpha t+ \alpha \cos \alpha t
			\end{bmatrix}.
		\end{aligned}
	\end{equation}
	We have
	\begin{equation}
		\begin{aligned}
			\left(Z(0)\right)^{-1} 
			=&\, \frac{1}{\alpha} \begin{bmatrix}
				\alpha 	 & 0\\  \frac{\delta \omega_0}{2} \quad &1
			\end{bmatrix}.
		\end{aligned}
	\end{equation}
	Let $X(t) = Z(t) (Z(0))^{-1}$, we have
	\begin{equation}
		\begin{aligned}
			X(t) =&\,   \frac{1}{\alpha} e^{-\delta \omega_0 t/2} \begin{bmatrix}
				\cos \alpha t & \sin \alpha t \\ - \frac{\delta \omega_0}{2} \cos \alpha t - \alpha \sin \alpha t & -\frac{\delta \omega_0}{2} \sin \alpha t+ \alpha \cos \alpha t
			\end{bmatrix} \begin{bmatrix}
				\alpha & 0 \\ \frac{\delta \omega_0}{2} & 1
			\end{bmatrix} \\
			=&\, \frac{1}{\alpha} e^{- \delta \omega_0 t /2} \begin{bmatrix}
				\alpha \cos \alpha t + \frac{\delta \omega_0}{2} \sin \alpha t & \sin \alpha t\\
				( - \alpha^2 - \frac{\delta^2 \omega^2_0}{4} )\sin \alpha t \qquad & \alpha \cos \alpha t - \frac{\delta \omega_0}{2} \sin \alpha t
			\end{bmatrix}
		\end{aligned}
	\end{equation}
	and 
	\begin{equation}
		(X(t))^{-1} = \frac{1}{\alpha}e^{\delta \omega_0 t/2} \begin{bmatrix}
			\alpha \cos \alpha t - \frac{\delta \omega_0}{2} \sin \alpha t	 & -\sin \alpha t \\ 
			(\alpha^2 + \frac{\delta^2 \omega_0^2}{4})\sin \alpha t \quad & \alpha \cos \alpha t + \frac{\delta \omega_0}{2} \sin \alpha t
		\end{bmatrix}.
	\end{equation}
	Then again using the method of  variation of parameters, we find that the solution has the form
	\begin{equation}
		\begin{aligned}
			V(t) = X(t)\left(V(0)+ \intt \left(X(\tau)\right)^{-1}f(\tau)\dtau \right),
		\end{aligned}
	\end{equation}
	and
	\begin{align}
		\det (I- X(T)) &= 1+e^{-T \delta \omega_0}-2 e^{-T\frac{\delta \omega_0}{2}}\cos \alpha T = (e^{-T\frac{\delta \omega_0}{2}}-\cos \alpha T)^2+ \sin^2 \alpha T >0
	\end{align}
	since $\delta>0$. Since periodic solutions must satisfy $V(0)=V(T)$, we then have
	\begin{equation}
		\begin{aligned}
			V(0)= (I-X(T))^{-1}X(T) \int_0^T \left(X(\tau)\right)^{-1}f(\tau)\dtau,
		\end{aligned}
	\end{equation}
	and arrive at the desired estimate in the same manner as before.\qed
\end{proof}

\section{Proof of the cut-off error estimate \eqref{conv Fourier}} \label{Appendix: Cutoff est}
We provide here the proof of the error estimate
\begin{equation} \tag{\ref{conv Fourier}}
	\begin{aligned}
|\!|\!|(p-\tpN, v-\tvN)|\!|\!|_{\left(\Xplowzero \cap H^2(\Hone)\right)\times \Xvlow} 
		\leq\, C N^{-\ell} (\|p\|_{Y_p^{\ell}\cap H^{\ell+2}(0,T; \Hone)}+\|v\|_{Y_v^{\ell}}), \quad \ell \geq 1,
	\end{aligned}
\end{equation}
for $(p, v) \in \left(\Xplowell \cap H^{\ell+2}(0,T; \Hone)\right) \times \Xvlowell$, which follows by a straightforward modification of the arguments in~\cite[Lemma 12]{bachinger2005numerical}.

\begin{proof}[proof of \eqref{conv Fourier}]
	We can express the norms as follows:
	\begin{equation}
		\begin{aligned}
			\|p\|^2_{\Xplowzero  \cap H^2(\Hone) } =&\,\begin{multlined}[t] \int_0^T \Bigr(\|\nabla p\|^2_{\Ltwo} + \|\nabla \pt\|^2_{\Ltwo}+ \|\nabla \ptt\|^2_{\Ltwo}\\\
 +\|p\|^2_{\Ltwo}+ \|\pt\|^2_{\Ltwo} + \|\ptt\|^2_{\Ltwo} \\
 +\|p\|^2_{L^2(\partial \Omega)}+ \|\pt\|^2_{L^2(\partial \Omega)} + \|\ptt\|^2_{L^2(\partial \Omega)} \Bigr) \dt
				\end{multlined} \\
			=&\,\frac{T}{2} \begin{multlined}[t] \sum_{m=0}^\infty \Bigl\{ \left(1+ (m \omega)^2+ (m \omega)^4\right) \left(\|\nabla p_m^c\|^2_{\Ltwo}+\|\nabla p_m^s\|^2_{\Ltwo}\right)\\
				+\left(1+ (m \omega)^2+ (m \omega)^4\right) \left(\| p_m^c\|^2_{\Ltwo}+\| p_m^s\|^2_{\Ltwo} \right)\\
				+\left(1+ (m \omega)^2+ (m \omega)^4\right) \left(\| p_m^c\|^2_{L^2(\partial \Omega)}+\| p_m^s\|^2_{L^2(\partial \Omega)} \right)
				\Bigr\}
				\end{multlined}
		\end{aligned}
	\end{equation}
and, similarly,	
	\begin{equation}
	\begin{aligned}
		\|v\|^2_{\Xvlowzero }
		=&\,\frac{T}{2} \begin{multlined}[t] \sum_{m=0}^\infty \Bigl\{ 
			\left(1+ (m \omega)^2+ (m \omega)^4\right) \left(\| v_m^c\|^2_{\Ltwo}+\| v_m^s\|^2_{\Ltwo} \right)
			\Bigr\}.
		\end{multlined}
	\end{aligned}
\end{equation}
Therefore, the errors are given by
\begin{equation}
	\begin{aligned}
		\|p-\tpN\|^2_{\Xplowzero \cap H^2(\Hone)}=&\,\frac{T}{2} \begin{multlined}[t] \sum_{m=N+1}^\infty \Bigl\{ \left(1+ (m \omega)^2+ (m \omega)^4\right) \left(\|\nabla p_m^c\|^2_{\Ltwo}+\|\nabla p_m^s\|^2_{\Ltwo}\right)\\
			+\left(1+ (m \omega)^2+ (m \omega)^4\right) \left(\| p_m^c\|^2_{\Ltwo}+\| p_m^s\|^2_{\Ltwo} \right)\\
			+\left(1+ (m \omega)^2+ (m \omega)^4\right) \left(\| p_m^c\|^2_{L^2(\partial \Omega)}+\| p_m^s\|^2_{L^2(\partial \Omega)} \right)
			\Bigr\}
		\end{multlined}
	\end{aligned}
\end{equation}
and
	\begin{equation}
	\begin{aligned}
		\|v-\tvN\|^2_{\Xvlowzero }
		=&\,\frac{T}{2} \begin{multlined}[t] \sum_{m=N+1}^\infty \Bigl\{ 
			\left(1+ (m \omega)^2+ (m \omega)^4\right) \left(\| v_m^c\|^2_{\Ltwo}+\| v_m^s\|^2_{\Ltwo} \right)
			\Bigr\}.
		\end{multlined}
	\end{aligned}
\end{equation}
We then have the estimate
\begin{equation}
	\begin{aligned}
			&\|p-\tpN\|^2_{\Xplowzero  \cap H^2(\Hone)} \\
			\leq&\,   \frac{T}{2}\tilde{C} \max_{m \geq N+1} \frac{1}{1+(m \omega)^{2\ell}}\sum_{m=N+1}^\infty \Bigl\{ \left(1+ (m \omega)^2+\ldots+(m \omega)^{2\ell+4}\right) \left(\|\nabla p_m^c\|^2_{\Ltwo}+\|\nabla p_m^s\|^2_{\Ltwo}\right)\\
			&\hspace*{1cm}	+\left(1+ (m \omega)^2+ (m \omega)^4+\ldots+(m \omega)^{2\ell+4}\right) \left(\| p_m^c\|^2_{\Ltwo}+\| p_m^s\|^2_{\Ltwo} \right)\\
		&\hspace*{1.5cm}		+\left(1+ (m \omega)^2+ (m \omega)^4+\ldots+(m \omega)^{2\ell+4}\right) \left(\| p_m^c\|^2_{L^2(\partial \Omega)}+\| p_m^s\|^2_{L^2(\partial \Omega)} \right)
				\Bigr\}
	\end{aligned}
\end{equation}
for some $\tilde{C}>0$, independent of $N$. From here, we conclude that
\begin{equation}
	\begin{aligned}
		\|p-\tpN\|^2_{\Xplowzero  \cap H^2(\Hone)} 
		\leq&\,  \frac{T}{2} \tilde{C} \frac{1}{1+((N+1) \omega)^{2\ell}} \|p\|^2_{\Xplow^\ell \cap H^{\ell+2}(\Hone)}\\[1mm] 
		\leq&\, C N^{-2\ell} \|p\|^2_{\Xplow^\ell \cap H^{\ell+2}(\Hone)}.
	\end{aligned}
\end{equation}
for some $C>0$, independent of $N$. We can analogously derive the estimate
\begin{equation}
	\begin{aligned}
		\|v-\tvN\|^2_{\Xplowzero }  \leq \tilde{C} N^{-2\ell} \|v\|^2_{\Xvlow^\ell},
	\end{aligned}
\end{equation}
to conclude the proof.\qed
\end{proof}

\section{Proof of estimate~\eqref{error est}}  \label{Appendix: Proof error energy est}
We provide here the proof of \eqref{error est}. We start from
 \begin{equation} \label{1}
	\begin{aligned}
		<\calL_1 \errpN, \phiN> 
		=&\, \begin{multlined}[t]-\intTO \fN \phiN \dxs + <\calL_1 \errprojNp, \phiN>
		\end{multlined}
	\end{aligned}
\end{equation}	
with the short-hand notation
\begin{equation}
	\fN = \eta (p^2-(\pNone)^2)_{tt} + c^2 \rho_0 n_0 (\vtt-\vttN).
\end{equation}
The proof follows by testing \eqref{1} with $\errpN$, $\errpN_t$, and $\errpN_{tt}$ and exploiting the time-periodic conditions. \\[2mm]

\noindent \underline{Step I}: Testing \eqref{1} with $\errpN$ leads to
\begin{equation}
	\begin{aligned}
		&c^2 \|\nabla \errpN\|^2_{\LtwoLtwo} + c^2\gamma \|\errpN\|^2_{L^2(L^2(\partial \Omega))} \\
		=&\, \begin{multlined}[t]
			- \intTO \errpN_{tt} \errpN \dxs 	- \intT \int_{\partial \Omega} b \beta \errpN_{tt} \errpN \dGs 
			-\intTO \fN \errpN \dxs \\+ <\calL_1 \errprojNp, \errpN>.
		\end{multlined}
	\end{aligned}
\end{equation}
Above, we have relied on the fact that
\begin{equation}
	 \int_0^T \int_{\Omega} b \nabla \errpN_t \cdot \nabla \errpN \dxs =0 \quad \text{and} \quad  \intT \int_{\partial \Omega}(  c^2 \beta+b \gamma)\errpN_t \errpN \dGs =0
\end{equation}
due to time-periodic conditions. We can further use integration by parts in time together with $\errpN(0)=\errpN(T)$ and $\errpN_t(0)=\errpN_t(T)$ to conclude that
\begin{equation}
	\begin{aligned}
			&- \intTO \errpN_{tt} \errpN \dxs 	- \intT \int_{\partial \Omega} b \beta \errpN_{tt} \errpN \dGs \\
			=&\,  \|\errpN_{t}\|^2_{\LtwoLtwo} 	+  b \beta \|\errpN_{t}\|^2_{L^2(L^2(\Omega))}.
	\end{aligned}
\end{equation}
Furthermore, we have 
\begin{equation}
	\begin{aligned}
		&<\calL_1 \errprojNp, \errpN> \\
		\coloneqq&\, \begin{multlined}[t]	\intTO \left((\errprojNp)_{tt} \errpN + c^2 \nabla \errprojNp \cdot \nabla \errpN+ b \nabla (\errprojNp)_t \cdot \nabla \errpN \right)\dxs \\
		+ \intT \int_{\partial \Omega} (c^2(\beta (\errprojNp)_t+\gamma \errprojNp)+b(\beta (\errprojNp)_{tt}+\gamma (\errprojNp)_t))\errpN \dGs. 
	\end{multlined}
	\end{aligned}
\end{equation}
and thus using H\"older's and Young's inequalities leads to
\begin{equation}
	\begin{aligned}
&\left|	<\calL_1 \errprojNp, \errpN> \right| \\
\lesssim&\, \begin{multlined}[t]
\eps \|\errpN\|^2_{\LtwoLtwo} + \eps \|\nabla \errpN\|^2_{\LtwoLtwo} + \eps \|\nabla \errpN_t\|^2_{\LtwoLtwo} +\eps \|\errpN\|^2_{L^2(L^2(\partial \Omega))}\\
+	\|(\errprojNp)_{tt}\|^2_{L^2(\Ltwo)} +	\|\nabla \errprojNp \|^2_{H^1(\Ltwo)}
+\|\errprojNp\|^2_{H^2(L^2(\partial \Omega))}
	\end{multlined}
	\end{aligned}
\end{equation}
for any $\eps>0$. Similarly,
\begin{equation}
	\left| -\intTO \fN \errpN \dxs \right| \lesssim \|\fN\|^2_{\LtwoLtwo} +\eps \|\errpN\|^2_{\LtwoLtwo}.
\end{equation}
Thus, for $\eps$ sufficiently small, the outcome of testing in Step I can be formulated as
\begin{equation}
	\begin{aligned}
		 \|\nabla \errpN\|^2_{\LtwoLtwo} + \|\errpN\|^2_{L^2(L^2(\partial \Omega))} 
	\lesssim&\, \begin{multlined}[t]
 \eps \|\errpN\|^2_{\LtwoLtwo}  + \eps \|\nabla \errpN_t\|^2_{\LtwoLtwo} \\
 + \|\fN\|^2_{\LtwoLtwo}+ \|\errprojNp\|^2_{\Xplowzero}.
		\end{multlined}
	\end{aligned}
\end{equation}
Further using the equivalence of $\|\nabla w\|^2_{\Ltwo}+\|w\|^2_{L^2(\partial \Omega)}$ and $\|w\|^2_{\Hone}$ yields for small $\eps$
\begin{equation}
	\begin{aligned}
		\| \errpN\|^2_{L^2(\Hone)} 
		\lesssim&\, \begin{multlined}[t]
			 \eps \|\nabla \errpN_t\|^2_{\LtwoLtwo} 
			+ \|\fN\|^2_{\LtwoLtwo}+ \|\errprojNp\|^2_{\Xplowzero}.
		\end{multlined}
	\end{aligned}
\end{equation}
~\\
\noindent \underline{Step II}: Testing \eqref{1} with $\errpN_t$ and exploiting time-periodicity leads to
\begin{equation}
	\begin{aligned}
		&b \|\nabla \errpN_t\|^2_{\LtwoLtwo} + (c^2+\gamma)\beta \|\errpN_t\|^2_{L^2(L^2(\partial \Omega))} \\
		=&\, \begin{multlined}[t]
			-\intTO \fN \errpN_t \dxs + <\calL_1 \errprojNp, \errpN_t>.
		\end{multlined}
	\end{aligned}
\end{equation}
Then proceeding similarly to Step I results in 
\begin{equation}
	\begin{aligned}
	 \| \errpN_t\|^2_{L^2(\Hone)} \lesssim  \|\fN\|^2_{\LtwoLtwo}+ \|\errprojNp\|^2_{\Xplowzero}.
	\end{aligned}
\end{equation}
~\\
\noindent \underline{Step III}: Testing \eqref{1} with $\errpN_{tt}$ and exploiting time-periodicity leads to
\begin{equation}
	\begin{aligned}
		&b \| \errpN_{tt}\|^2_{\LtwoLtwo} + b \beta \|\errpN_{tt}\|^2_{L^2(L^2(\partial \Omega))} \\
		=&\, \begin{multlined}[t]
			- c^2 \intTO \nabla \errpN \cdot \nabla \errpN_{tt} - c^2 \gamma \intT \int_{\partial \Omega}  \errpN  \errpN_{tt} \dGs 
			-\intTO \fN \errpN_{tt} \dxs \\+ <\calL_1 \errprojNp, \errpN_{tt}>.
		\end{multlined}
	\end{aligned}
\end{equation}
By time-periodicity, we have
\begin{equation}
	\begin{aligned}
	& - c^2 \intTO \nabla \errpN \cdot \nabla \errpN_{tt} - c^2 \gamma \intT \int_{\partial \Omega}  \errpN  \errpN_{tt} \dGs\\
	=&\, c^2 \|\nabla \errpN_t\|^2_{\LtwoLtwo} +c^2 \gamma \|\errpN_t\|^2_{L^2(L^2(\partial \Omega))}.
	\end{aligned}
\end{equation}
Furthermore, we have 
\begin{equation}
	\begin{aligned}
		&<\calL_1 \errprojNp, \errpN_{tt} > \\
		\coloneqq&\, \begin{multlined}[t]	\intTO \left((\errprojNp)_{tt} \errpN_{tt} + c^2 \nabla \errprojNp \cdot \nabla \errpN_{tt} + b \nabla (\errprojNp)_t \cdot \nabla \errpN_{tt}  \right)\dxs \\
			+ \intT \int_{\partial \Omega} (c^2(\beta (\errprojNp)_t+\gamma \errprojNp)+b(\beta (\errprojNp)_{tt}+\gamma (\errprojNp)_t))\errpN_{tt}  \dGs. 
		\end{multlined}
	\end{aligned}
\end{equation}
We integrate by parts in time in the $c^2$ and $b$ terms on the left-hand side:
\begin{equation}
	\begin{aligned}
	& \intTO \left( c^2 \nabla \errprojNp \cdot \nabla \errpN_{tt} + b \nabla (\errprojNp)_t \cdot \nabla \errpN_{tt}  \right)\dxs \\
	=&\, -\intTO \left( c^2 \nabla (\errprojNp)_t \cdot \nabla \errpN_{t} + b \nabla (\errprojNp)_{tt} \cdot \nabla \errpN_{t}  \right)\dxs 
	\end{aligned}
\end{equation}
and then estimate
\begin{equation}
	\begin{aligned}
		\left| <\calL_1 \errprojNp, \errpN_{tt} > \right| 	\lesssim&\, \begin{multlined}[t] \eps \|\errpN_{tt}\|^2_{\LtwoLtwo} + \eps \|\nabla \errpN_t\|^2_{\LtwoLtwo}+ \eps \|\errpN_{tt}\|^2_{L^2(L^2(\partial \Omega))}  \\
			+ \|(\errprojNp)_{tt}\|^2_{\LtwoLtwo}+ \|\nabla (\errprojNp)_{t}\|^2_{\LtwoLtwo}\\ +  \|\nabla (\errprojNp)_{tt}\|^2_{\LtwoLtwo} 
			+\|\errprojNp\|^2_{H^2(\partial \Omega)}.
		\end{multlined}
	\end{aligned}
\end{equation}
The presence of the term $\|\nabla (\errprojNp)_{tt}\|^2_{\LtwoLtwo} $ above is the reason for assuming $p \in H^{\ell+2}(0,T; \Hone)$ (in addition to $p \in \Xp \cap \Xplowell$). We also have
\begin{equation}
	\left| -\intTO \fN \errpN_{tt} \dxs \right| \lesssim \|\fN\|^2_{\LtwoLtwo} + \eps \|\errpN_{tt}\|^2_{\LtwoLtwo}.
\end{equation}
Thus, the outcome of testing in Step III for small enough $\eps$ is
\begin{equation}
	\begin{aligned}
		&\| \errpN_{tt}\|^2_{\LtwoLtwo} + \|\errpN_{tt}\|^2_{L^2(L^2(\partial \Omega))} \\
		\lesssim&\, \begin{multlined}[t]
			\eps \|\nabla \errpN_t\|^2_{\LtwoLtwo} +  \|\fN\|^2_{\LtwoLtwo}+\|\errprojNp\|^2_{\Xplowzero \cap H^2(\Hone)}.
		\end{multlined}
	\end{aligned}
\end{equation}
Combining the three obtained estimates leads to \eqref{error est}, provided $\eps$ is sufficiently small. \qed

\bibliography{references}{}
\bibliographystyle{siam} 
\end{document}